\documentclass[11pt]{amsart}

\usepackage[T1]{fontenc}
\usepackage[utf8]{inputenc}
\usepackage[english]{babel}
\usepackage{lmodern}

\usepackage{amsmath,amssymb,mathtools}
\usepackage{mathrsfs}
\usepackage{bm}

\usepackage{enumitem}
\usepackage{microtype}

\usepackage[hidelinks]{hyperref}
\usepackage[nameinlink,noabbrev]{cleveref}

\newtheorem{theorem}{Theorem}[section]
\newtheorem{proposition}[theorem]{Proposition}
\newtheorem{lemma}[theorem]{Lemma}
\newtheorem{corollary}[theorem]{Corollary}

\theoremstyle{definition}
\newtheorem{definition}[theorem]{Definition}

\theoremstyle{remark}
\newtheorem{remark}[theorem]{Remark}
\newtheorem{example}[theorem]{Example}

\newtheorem*{theoremA}{Theorem A}
\newtheorem*{theoremB}{Theorem B}
\newtheorem*{theoremC}{Theorem C}
\newtheorem*{factorizationlemma}{Lemma}

\newcommand{\C}{\mathbb C}
\newcommand{\R}{\mathbb R}
\newcommand{\T}{\mathbb T}
\newcommand{\Z}{\mathbb Z}

\newcommand{\one}{\mathbf 1}
\newcommand{\norm}[1]{\left\lVert #1\right\rVert}
\newcommand{\abs}[1]{\left\lvert #1\right\rvert}

\newcommand{\fhat}{\widehat{\!f}}
\newcommand{\supp}{\operatorname{supp}}

\newcommand{\Ind}{\operatorname{Ind}}

\title[Norm-Controlled Inversion in Measure Algebras]{Norm-Controlled Inversion in Measure Algebras}
\author{Przemys{\l}aw Ohrysko}
\address{University of Warsaw\\
Faculty of Mathematics, Informatics and Mechanics\\
Institute of Mathematics\\
Banacha 2, 02-097 Warsaw, Poland}
\email{P.Ohrysko@mimuw.edu.pl}
\date{September 2026}

\subjclass[2020]{Primary 43A10; Secondary 43A25, 46J10, 47B35}
\keywords{Measure algebra, Fourier algebra, Wiener algebra, norm-controlled inversion, Toeplitz operator, Wiener--Pitt phenomenon}

\begin{document}

\begin{abstract}
We solve Nikolski's norm-controlled inversion problem for measure algebras of
locally compact Abelian groups. For every $\delta>1/2$ there is a constant
$C_M(\delta)$, independent of the group $G$, such that
$\|\mu\|_{M(G)}\le1$ and
$\inf_{\gamma\in\widehat G}|\widehat\mu(\gamma)|\ge\delta$ imply that $\mu$ is
invertible in $M(G)$ and
$\|\mu^{-1}\|_{M(G)}\le C_M(\delta)$. Together with Nikolski's negative results,
this shows that $1/2$ is the sharp universal threshold. The main step is a
uniform inversion theorem for Fourier algebras $A(K)$ of compact Abelian groups;
its proof is based on a profile decomposition and a resolvent argument. We also obtain norm-controlled inversion for Toeplitz operators with non-vanishing
Wiener symbols of winding number zero and
give direct endpoint constructions showing failure of norm control at
$\delta=1/2$ in $A(K)$ for every infinite compact Abelian group $K$, and in the
unitization of $L^1(G)$ for every nondiscrete locally compact Abelian group $G$.
\end{abstract}

\maketitle

\section{Introduction}
\label{sec:introduction}

\subsection{From Wiener's theorem to quantitative inversion}

Wiener's $1/f$ theorem is one of the basic inversion principles of harmonic
analysis. If
\[
  f(t)=\sum_{n\in\Z}a_ne^{int},
  \qquad
  \sum_{n\in\Z}|a_n|<\infty,
\]
and $f$ has no zeros on the circle, then $1/f$ again has an absolutely
convergent Fourier series. Equivalently, the Wiener algebra $A(\T)$ is inverse-closed in $C(\T)$. Several
proofs of this qualitative statement illuminate different aspects of the quantitative
problem considered here.

Wiener's original argument \cite{Wiener1932} is local in nature. The reciprocal
can be constructed on pieces of the range on which $f$ stays away from zero,
and the local information is patched by the algebraic and approximation
properties of absolutely convergent Fourier series. After the development of
commutative Banach-algebra theory, the same theorem acquired a much shorter
spectral interpretation: the maximal ideal space of $A(\T)$ is $\T$, so the
Gelfand criterion says directly that pointwise nonvanishing is equivalent to
invertibility in the algebra; see \cite{Gelfand1941,Kaniuth2009,Zelazko1973}.
Thus the classical theorem may be read either as a concrete Fourier-analytic
localization theorem or as a statement identifying the visible points of the
spectrum with the whole maximal ideal space.

There is also a third line of proofs, closer in spirit to effective inversion.
Beurling gave a direct proof based on spectral-radius estimates
\cite{Beurling1939}; Newman later found a particularly short elementary proof
\cite{Newman1975}; and Cohen developed a constructive Banach-algebraic method
that avoids the maximal-ideal argument \cite{Cohen1961}. These approaches are
important conceptually because they show that inversion can be produced without
first identifying every maximal ideal. They do not, however, answer the question
that is central here: if both the algebra norm of $f$ and its distance from zero
are prescribed, can one bound the norm of $f^{-1}$ uniformly? In other words,
qualitative inverse-closedness and norm-controlled inversion are genuinely
different assertions.

The distinction becomes much sharper for measure algebras. Let $G$ be a locally
compact Abelian group, written additively, with dual group $\widehat G$ and Haar
measure $m_G$. We write $M(G)$ for the Banach algebra of complex regular Borel
measures on $G$, with convolution and total variation norm; its identity is
$\delta_0$. For $\mu\in M(G)$,
\[
  \widehat\mu(\gamma)
  :=\int_G\overline{\gamma(x)}\,d\mu(x),
  \qquad \gamma\in\widehat G.
\]
We denote by $M_d(G)$ the closed subalgebra of discrete measures and by $M_c(G)$
the closed subspace of continuous, that is, nonatomic, measures. For a commutative
unital Banach algebra $\mathcal A$, we write $\Delta(\mathcal A)$ for its maximal
ideal space. If $G$ is nondiscrete, then the
Fourier--Stieltjes transform does not see the whole maximal ideal space: there
are measures $\mu\in M(G)$ such that
\[
  \inf_{\gamma\in\widehat G}|\widehat\mu(\gamma)|>0
\]
while $\mu$ is not invertible in $M(G)$. This is the Wiener--Pitt phenomenon,
originating in \cite{WienerPitt1938}; its spectral structure was developed by
\v{S}re\u{\i}der \cite{Shreider1950}, and we refer also to
\cite{Williamson1959,Graham1974,OhryskoSandersWojciechowski2026} for related
developments. In Banach-algebraic language,
$\widehat G$ embeds canonically into $\Delta(M(G))$, but for nondiscrete $G$ this
visible part is too small to detect invertibility.

We shall repeatedly use two standard companion algebras. We identify
$L^1(G)=L^1(G,m_G)$ with the corresponding closed ideal of absolutely continuous
measures in $M(G)$. Its Gelfand space is canonically $\widehat G$ and its Gelfand
transform is the ordinary Fourier transform; see
\cite[Theorem~1.2.2]{Rudin1990}. We also use the Riemann--Lebesgue property
$\widehat f\in C_0(\widehat G)$ for $f\in L^1(G)$
\cite[Theorem~1.2.4(a)]{Rudin1990}. If $K$ is compact, Haar measure is
normalized by $m_K(K)=1$, and
\[
  A(K)
  :=\left\{f\in C(K):
  \sum_{\gamma\in\widehat K}|\widehat f(\gamma)|<\infty\right\},
  \qquad
  \|f\|_{A(K)}:=\sum_{\gamma\in\widehat K}|\widehat f(\gamma)|.
\]
The Fourier transform is an isometric Banach-algebra isomorphism
$A(K)\cong\ell^1(\widehat K)$, where the latter algebra is equipped with
convolution. For nondiscrete $G$ we use the standard unitization
\[
  L^1(G)^\#=\C\one\oplus_1L^1(G),
  \qquad
  \|\lambda\one+f\|_{L^1(G)^\#}=|\lambda|+\|f\|_1,
\]
realized in $M(G)$ as $\C\delta_0\oplus_1L^1(G)$. If $G$ is discrete, then
$L^1(G)=\ell^1(G)$ is already unital. For an Abelian group $\Gamma$, the notation
$\Gamma_d$ means the same group endowed with the discrete topology.

\subsection{Nikolski's quantitative framework and the critical gap}

The systematic quantitative formulation used in this paper goes back to
Nikolski's 1999 work \emph{In search of the invisible spectrum}
\cite{Nikolski1999}. Nikolski distinguished the visible spectrum from the full Gelfand spectrum and
introduced optimal inversion majorants and critical constants to quantify the resulting
obstruction. This point of view turned the Wiener--Pitt phenomenon into a quantitative
inversion problem.

Following \cite[Subsection~0.2]{Nikolski1999}, let $\mathcal A$ be a commutative
unital Banach algebra and let $X$ be a specified visible part of its maximal ideal
space. Assume that restriction of the Gelfand transform defines a continuous embedding
of $\mathcal A$ into $C_b(X)$, equipped with the supremum norm. For
$0<\delta\le1$, define
\[
  c_1(\delta;\mathcal A,X)
  :=\sup\left\{
    \|a^{-1}\|_{\mathcal A}:
    \|a\|_{\mathcal A}\le1,
    \ \inf_{x\in X}|a(x)|\ge\delta
  \right\},
\]
We set $c_1(\delta;\mathcal A,X)=+\infty$ if any admissible element is not
invertible. The first critical constant is
\[
  \delta_1(\mathcal A,X)
  :=\inf\{0<\delta\le1:c_1(\delta;\mathcal A,X)<\infty\}.
\]
Thus the problem has two logically distinct parts: determine when the visible
lower bound forces invertibility, and determine when it forces a uniform bound
on the inverse norm.

Norm-controlled inversion has also been studied in other inverse-closed Banach
algebra settings, including smooth Banach subalgebras and weighted convolution
algebras; see, for example, \cite{GrochenigKlotz2013,SameiShepelska2019}. A
recent result \cite{Ohrysko2026} establishes norm-controlled inversion for two
classes of algebras of integrable functions with summable Fourier transforms.
That setting is structurally different from the present measure-algebra problem:
here the visible Fourier--Stieltjes transform may fail to detect the full spectrum,
and the main issue is to obtain a bound uniform in the underlying locally compact
Abelian group.

Nikolski placed this formulation in a longer historical line. Stafney had shown
that inverse norms in the Wiener algebra can be unbounded even under simultaneous
uniform control of the original function and its range, and Shapiro obtained a
constructive counterexample of the same general type; see
\cite{Stafney1967,Shapiro1979} and Nikolski's historical discussion in
\cite[Subsection~0.5]{Nikolski1999}. Those results showed that a qualitative
Wiener theorem does not automatically admit a quantitative counterpart, but they
did not identify the sharp lower spectral threshold. Nikolski's framework made
that threshold itself an invariant to be estimated.

For $M(G)$ the visible set is $\widehat G$. Nikolski proved the universal upper
estimate
\begin{equation}
  c_1(\delta;M(G),\widehat G)
  \le \frac{1}{2\delta^2-1},
  \qquad
  \delta>\frac1{\sqrt2},
  \label{eq:intro-nikolski-upper}
\end{equation}
and showed that no universal norm control is possible at or below $1/2$; see
\cite[Theorem~2.1.1, Corollary~3.2.5 and Subsection~2.4.3]{Nikolski1999}.
The lower obstruction is already present for the Wiener algebra of every
infinite compact Abelian group. Nikolski's proof of the lower estimate uses a passage to the Bohr
compactification; this mechanism
will reappear in a much simpler role below when we transfer the positive result
from Fourier algebras to measure algebras.

It is instructive to isolate the mechanism behind the positive estimate
\eqref{eq:intro-nikolski-upper}, because it explains both its power and its
limitation. In the measure-algebra setting, with the standard involution
$\widetilde\mu(E):=\overline{\mu(-E)}$, one first passes to an element such as
$\mu*\widetilde\mu$, whose Fourier--Stieltjes transform is nonnegative. The LCA-group version of Wiener's mean-square formula for Fourier--Stieltjes
transforms, in the form used in \cite[proof of Theorem~2.1.1]{Nikolski1999}, gives
\[
  (\mu*\widetilde\mu)(\{0\})
  =\sum_{x\in G}|\mu(\{x\})|^2
  \ge \delta^2.
\]
Write $\eta=\mu*\widetilde\mu=a\delta_0+r$, where
$a=\eta(\{0\})\ge\delta^2$. Since $\norm{\eta}\le1$, one has
$\norm r\le1-a<a$ whenever $\delta>1/\sqrt2$. Nikolski's elementary
inversion lemma \cite[Lemma~1.4.3]{Nikolski1999}, or directly the corresponding
geometric series, yields
\[
  \norm{\eta^{-1}}\le\frac1{2a-1}
  \le\frac1{2\delta^2-1}.
\]
Finally, $\mu^{-1}=\widetilde\mu*\eta^{-1}$, so the same bound holds for
$\mu^{-1}$. Thus the threshold $1/\sqrt2$ is exactly the point at which the
available lower bound $\delta^2$ forces a single coefficient to carry more than
half of the total norm.

The work \cite{OhryskoWasilewski2020} pushed this program in two different
directions. First, it solved the corresponding \emph{qualitative} inversion
problem above the conjectured threshold: if $\|\mu\|\le1$ and
$\inf_{\widehat G}|\widehat\mu|>1/2$, then $\mu$ is invertible. The key ingredients
are the Glicksberg--Wik theorem for the discrete part and the Bohr compactification
of the dual. This showed that $1/2$ is the correct threshold for invertibility,
but it did not yet give a group-independent bound for $\|\mu^{-1}\|$ throughout
that range.

Second, the same paper improved the quantitative estimates in important classes.
Writing the atomic part as
\[
  \mu_d=\sum_{n\ge1}a_n\delta_{\tau_n},
  \qquad |a_1|\ge|a_2|\ge\cdots,
\]
the argument combines Parseval-type information with the geometry of the first
two atoms to force the largest atom $|a_1|$ above $1/2$. Under the relevant
infinite-order hypothesis this yields norm-controlled inversion for
\[
  \delta>\frac{\sqrt{33}-1}{8}
\]
and an explicit majorant; see \cite[Theorem~2.8]{OhryskoWasilewski2020}. For a
compact connected Abelian group the dual is torsion-free, so the hypothesis is
automatic. In the Fourier-algebra formulation one obtains
\[
  c_1(\delta;A(K),K)
  \le
  \frac{2}{\sqrt{17\delta^2+6\delta-7}-(1+\delta)}.
\]
The improvement is substantial, but the final inversion step is still the same:
once one distinguished atom or Fourier coefficient is larger than $1/2$, the
remainder has norm less than that coefficient and an elementary Neumann-series
argument applies.

The quantitative arguments discussed above reduce inversion to the existence of one
sufficiently large distinguished coefficient, possibly after symmetrization or another
auxiliary transformation. Section~\ref{sec:positive}
starts by pushing this strategy further in the connected case. Theorem~\ref{thm:improved-largest-coefficient} proves that under
$\|f\|_{A(K)}\le1$ and $|f|\ge\delta>1/2$,
\begin{equation}
  \|\widehat f\|_\infty
  \ge
  \frac{\delta+\sqrt{2\delta-1}}{2}.
  \label{eq:intro-largest-coefficient}
\end{equation}
Consequently one obtains an explicit inverse estimate for
$\delta>2-\sqrt2$. However, Example~\ref{ex:no-large-fourier-coefficient} shows
that the assumptions $\|f\|_{A(\T)}\le1$ and $\inf_{\T}|f|>1/2$ do not force
\emph{any} Fourier coefficient of the original function to have modulus greater
than $1/2$. Thus a direct argument whose decisive step is to extract a single
dominant coefficient from $f$ itself cannot reach the sharp threshold. The
example does not rule out every possible auxiliary transformation leading to a
different large-coefficient argument.

The main new idea of the present paper is to replace that one-coefficient
mechanism by a \emph{profile method}. From a hypothetical bad sequence we extract all
coherent concentration profiles and write
\[
  f_j=h_j+w_j,
\]
where $h_j$ is the \emph{profile core} and the largest Fourier coefficient of the
remainder $w_j$ tends to zero. If $P$ is the total Wiener mass of the profiles and
$\beta$ is the limiting Wiener mass of the remainder, then the limiting-core
construction gives
\[
  P+\beta\le1,
  \qquad
  P\ge\delta,
  \qquad
  \beta\le1-\delta<\delta.
\]
This \emph{mass gap} is the point at which the threshold $1/2$ enters. Choosing
$\beta<\rho<\delta$ and introducing an auxiliary circle variable $z$, we invert a
limiting resolvent whose inverse has \emph{one-sided spectrum}, meaning only
non-negative powers of $z$. That information controls all inverse powers of $h_j$
simultaneously and allows the diffuse remainder to be absorbed. The full roadmap of
this argument is given immediately after Theorem~\ref{thm:uniform-fourier-inversion}.

\subsection{Main results and organization}

Our first main theorem is the resulting group-independent inversion theorem for
Fourier algebras of compact Abelian groups.

\begin{theoremA}
For every $\delta>1/2$ there is a constant $C_A(\delta)<\infty$ such that, for
every compact Abelian group $K$ and every $f\in A(K)$ with
\[
  \|f\|_{A(K)}\le1,
  \qquad
  \inf_{x\in K}|f(x)|\ge\delta,
\]
one has $f^{-1}\in A(K)$ and
\[
  \|f^{-1}\|_{A(K)}\le C_A(\delta).
\]
\end{theoremA}

This is proved as Theorem~\ref{thm:uniform-fourier-inversion}. The constant is
uniform in the group, but the proof is non-effective. The profile decomposition
and the limiting resolvent described above are the essential new ingredients.
The theorem is sharp for the class of infinite compact Abelian groups; its
sharpness will be discussed in Section~\ref{sec:sharpness}.

The Fourier-algebra theorem is then transferred to measure algebras.

\begin{theoremB}
For every $\delta>1/2$ there is a constant $C_M(\delta)<\infty$ such that, for
every locally compact Abelian group $G$ and every $\mu\in M(G)$ with
\[
  \|\mu\|_{M(G)}\le1,
  \qquad
  \inf_{\gamma\in\widehat G}|\widehat\mu(\gamma)|\ge\delta,
\]
the measure $\mu$ is invertible in $M(G)$ and
\[
  \|\mu^{-1}\|_{M(G)}\le C_M(\delta).
\]
\end{theoremB}

Theorem~\ref{thm:uniform-measure-inversion} proves this statement. The
Glicksberg--Wik theorem transfers the visible lower bound to the discrete part;
the latter is naturally a Fourier-algebra element on the Bohr compactification of
$\widehat G$. A budgeted version of the profile-resolvent argument controls all
inverse powers simultaneously, after which the continuous part is summed by an
absolutely convergent resolvent series.

The new point of Theorems~A and B is that the inverse bounds are uniform in the
underlying group throughout the full range $\delta>1/2$; together with the
endpoint and negative results discussed below, this identifies $1/2$ as the sharp
universal threshold.

The appearance of the constant $1/2$ in both Theorems~A and B is not a coincidence.
In the measure-algebra setting, $1/2$ already appears as the limiting level in the
qualitative inversion theorem of Ohrysko and Wasilewski: if $\|\mu\|\le1$ and
the visible Fourier--Stieltjes transform stays strictly above $1/2$, then $\mu$ is
invertible \cite[Theorem~2.4]{OhryskoWasilewski2020}. In the present paper, the same
number turns out to be the critical constant for norm-controlled inversion in
Theorem~A. The bridge between the two phenomena is provided by Nikolski's
``walk to the Bohr compact'' \cite[Subsections~3.1--3.2]{Nikolski1999}, which is
briefly revisited in Section~\ref{sec:sharpness}.

The same profile-resolvent mechanism also applies to Toeplitz operators on the
Wiener algebra, a problem considered by Nikolskaia \cite{Nikolskaia2001}. For the
statement below, let $W:=A(\T)$, let $W^+$ be the subalgebra of functions with no
negative Fourier frequencies, let $P_+:W\to W^+$ be the Riesz projection, and set
$T_fg:=P_+(fg)$. For a non-vanishing $f\in W$, let $\Ind(f)$ denote its winding
number around the origin.

\begin{theoremC}
For every $\delta>1/2$ there is a constant $C_T(\delta)<\infty$ such that,
whenever $f\in W$ satisfies $\|f\|_W\le1$, $\inf_{\T}|f|\ge\delta$, and
$\Ind(f)=0$, the operator $T_f:W^+\to W^+$ is invertible and
\[
  \|T_f^{-1}\|_{\mathcal L(W^+)}\le C_T(\delta).
\]
\end{theoremC}

This is Theorem~\ref{thm:toeplitz-half-threshold}. The passage is not formal,
because the inverse of an invertible Toeplitz operator need not be Toeplitz; an
additional logarithmic step is required in the limiting profile core.

A complementary explicit estimate survives from the largest-coefficient method.
For a compact connected Abelian group, Theorem~\ref{thm:improved-largest-coefficient}
yields \eqref{eq:intro-largest-coefficient}. Consequently, if
$\delta>2-\sqrt2$, then
\[
  \|f^{-1}\|_{A(K)}
  \le
  \frac{1}{\delta+\sqrt{2\delta-1}-1};
\]
see Theorem~\ref{thm:inverse-bound-connected}. This part of the paper is effective
and complements the non-effective full-range theorem.

Section~\ref{sec:positive} develops the positive results. Subsection~\ref{subsec:largest-coefficient}
first records an improved explicit estimate obtained from the dominant-coefficient
strategy. Subsection~\ref{subsec:fourier-algebra-positive} introduces the profile
decomposition and proves Theorem~A. Subsection~\ref{subsec:measure-algebra-positive}
transfers the argument to measure algebras and proves Theorem~B, while
Subsection~\ref{subsec:toeplitz-positive} gives the Toeplitz application.
Section~\ref{sec:sharpness} gives direct constructions showing that $1/2$ is the
sharp universal threshold: norm control fails at the endpoint in $A(K)$ for every
infinite compact Abelian group $K$ and in $L^1(G)^\#$ for every nondiscrete
locally compact Abelian group $G$. These endpoint constructions are independent of
the new positive profile argument. We conclude in
Section~\ref{sec:further-questions} with consequences and open problems.

\section{Uniform inversion above the critical threshold}
\label{sec:positive}

\subsection{Large Fourier coefficients in the connected case}
\label{subsec:largest-coefficient}

Throughout this subsection, $G$ denotes a compact connected Abelian group, $\widehat G$
its discrete dual group, and $m_G$ the normalized Haar measure on $G$. We use the
standard realization
\[
  A(G)=\left\{f=\sum_{\gamma\in\widehat G}\fhat(\gamma)\gamma:
  \sum_{\gamma\in\widehat G}\abs{\fhat(\gamma)}<\infty\right\},
  \qquad
  \norm{f}_{A(G)}=\sum_{\gamma\in\widehat G}\abs{\fhat(\gamma)}.
\]
Since an element of $\ell^1(\widehat G)$ has at most countable support, the
moduli of the non-zero Fourier coefficients of $f\in A(G)$ can be arranged in
non-increasing order
\[
  x_1\ge x_2\ge x_3\ge\cdots\ge0,
\]
with repetitions according to multiplicity; if $f$ is a trigonometric polynomial,
we pad the finite list by zeros. In particular,
\[
  x_1=\norm{\fhat}_{\infty}.
\]

We begin by recording the estimate from \cite{OhryskoWasilewski2020} that will be
used below. It is convenient to state it directly in the Fourier-algebra language.
Recall that, since $G$ is compact, $\widehat G$ is discrete
\cite[Theorem~1.2.5]{Rudin1990}; with the standard Fourier-algebra norm, the Fourier
transform therefore identifies $A(G)$ isometrically with
$M(\widehat G)=\ell^1(\widehat G)$.

\begin{proposition}[{\cite[Proposition~2.6]{OhryskoWasilewski2020}}]
\label{prop:OW-two-largest}
Let $G$ be a compact Abelian group and let $f\in A(G)$ satisfy
\[
  \norm{f}_{A(G)}\le1,
  \qquad
  \abs{f(x)}\ge\delta>\frac12
  \quad (x\in G).
\]
If $x_1\ge x_2\ge\cdots$ is the decreasing rearrangement of the moduli of the
non-zero Fourier coefficients of $f$, then
\[
  x_1\ge\delta^2,
  \qquad
  x_1+x_2\ge\delta,
\]
where the second inequality is understood as trivial when $f$ has only one non-zero
Fourier coefficient.
\end{proposition}

For comparison, we also record the stronger estimate obtained in
\cite[Theorem~2.8]{OhryskoWasilewski2020}. The formulation below is its specialization
to the present setting. Indeed, by \cite[Theorem~2.5.6(c)]{Rudin1990}, connectedness of $G$ is
equivalent to torsion-freeness of $\widehat G$. Hence the quotient of any two
distinct characters has infinite order. The cited theorem is stated with a strict lower
bound. Under the present assumption $|f|\ge\delta$, apply it with every
$\delta'\in(1/2,\delta)$. Letting $\delta'\uparrow\delta$ and using the continuity of
the displayed bounds gives the formulation below; for either inverse estimate, choose
$\delta'$ above the corresponding threshold before passing to the limit.

\begin{theorem}[{\cite[Theorem~2.8]{OhryskoWasilewski2020}}]
\label{thm:OW-largest}
Let $G$ be a compact connected Abelian group and let $f\in A(G)$ satisfy
\[
  \norm{f}_{A(G)}\le 1,
  \qquad
  \abs{f(x)}\ge\delta>\frac12
  \quad (x\in G).
\]
Then
\[
  \norm{\fhat}_{\infty}
  \ge
  \Psi(\delta)
  :=
  \frac{1-\delta+\sqrt{17\delta^2+6\delta-7}}{4}
  \ge
  \frac{3\delta-1}{2}.
\]
Moreover, the same theorem gives the following estimates for the norm of the inverse:
\[
  \norm{f^{-1}}_{A(G)}
  \le
  \frac{1}{3\delta-2}
  \qquad
  \text{if }\delta>\frac23,
\]
and
\[
  \norm{f^{-1}}_{A(G)}
  \le
  \frac{2}{\sqrt{17\delta^2+6\delta-7}-(1+\delta)}
  \qquad
  \text{if }\delta>\frac{\sqrt{33}-1}{8}.
\]
\end{theorem}

We now improve this lower bound. The additional gain comes from using the full
$\ell^1$ constraint on the Fourier tail after the second largest coefficient has been
controlled.

\begin{lemma}
\label{lem:packing}
Let $b>0$ and let $(y_j)_{j\ge1}$ be a sequence satisfying
\[
  0\le y_j\le b,
  \qquad
  \sum_{j\ge1}y_j\le S,
  \qquad
  2b<S\le3b.
\]
Then
\[
  \sum_{j\ge1}y_j^2\le 2b^2+(S-2b)^2.
\]
\end{lemma}

\begin{proof}
For $N\ge1$ put
\[
  T_N:=\sum_{j=1}^{N}y_j.
\]
If $T_N\le2b$, then
\[
  \sum_{j=1}^{N}y_j^2\le bT_N\le2b^2.
\]
Suppose now that $T_N>2b$. We redistribute the mass of the finite vector
$(y_1,\ldots,y_N)$ without changing its sum or leaving the interval $[0,b]$.
If two coordinates lie strictly between $0$ and $b$, denote them by $u,v$ with
$0<u\le v<b$ and set
\[
  \varepsilon:=\min\{u,b-v\}>0.
\]
Replace the pair $(u,v)$ by $(u-\varepsilon,v+\varepsilon)$. The sum is unchanged,
both new coordinates belong to $[0,b]$, and
\[
  (u-\varepsilon)^2+(v+\varepsilon)^2-u^2-v^2
  =2\varepsilon(v-u)+2\varepsilon^2\ge0.
\]
Thus the sum of squares does not decrease, while at least one of these two
coordinates reaches $0$ or $b$. Repeating the operation terminates after finitely
many steps, with at most one coordinate strictly between $0$ and $b$.
Since $2b<T_N\le S\le3b$, the resulting vector is, up to a permutation,
\[
  (b,b,T_N-2b,0,\ldots,0),
\]
where the third coordinate equals $b$ if $T_N=3b$. Consequently,
\[
  \sum_{j=1}^{N}y_j^2
  \le2b^2+(T_N-2b)^2
  \le2b^2+(S-2b)^2.
\]
Thus the same upper bound holds for every partial sum. Letting $N\to\infty$
proves the assertion.
\end{proof}

\begin{theorem}
\label{thm:improved-largest-coefficient}
Let $G$ be a compact connected Abelian group, let $1/2<\delta\le1$, and let
$f\in A(G)$ satisfy
\[
  \norm{f}_{A(G)}\le1,
  \qquad
  \abs{f(x)}\ge\delta
  \quad (x\in G).
\]
Then
\[
  \norm{\fhat}_{\infty}
  \ge
  \Phi(\delta)
  :=
  \frac{\delta+\sqrt{2\delta-1}}{2}.
\]
\end{theorem}

\begin{proof}
If $\delta=1$, then $\norm{f}_\infty\le\norm{f}_{A(G)}\le1$ and $\abs f\ge1$, so
$\abs f=1$ on $G$. By Parseval, if $(x_j)$ denotes the decreasing rearrangement
of the moduli of the non-zero Fourier coefficients, then
\[
  1=\sum_j x_j^2\le\left(\sum_j x_j\right)^2\le1.
\]
Equality between the $\ell^2$- and $\ell^1$-norms forces exactly one $x_j$ to be
non-zero. Hence $\norm{\fhat}_\infty=1=\Phi(1)$, and the conclusion follows.
We may therefore assume $\delta<1$.

Let $x_1\ge x_2\ge\cdots$ be the decreasing rearrangement of the moduli of the
non-zero Fourier coefficients of $f$. By Parseval's identity (see \cite[Theorem~1.6.1]{Rudin1990}),
\[
  \sum_{j\ge1}x_j^2
  =\int_G\abs{f(x)}^2\,dm_G(x)
  \ge\delta^2.
\]
Moreover, Proposition~\ref{prop:OW-two-largest} gives
\begin{equation}
\label{eq:two-largest}
  x_1+x_2\ge\delta.
\end{equation}

If $f$ has only one non-zero Fourier coefficient, then
$\norm{\fhat}_{\infty}=x_1\ge\delta$, and $\delta\ge\Phi(\delta)$. We may therefore assume that at least two Fourier
coefficients are non-zero. Choose distinct $\gamma_1,\gamma_2\in\widehat G$ such that
\[
  \abs{\fhat(\gamma_1)}=x_1,
  \qquad
  \abs{\fhat(\gamma_2)}=x_2.
\]
Multiplying $f$ by a unimodular scalar and by the character
$\overline{\gamma_1}$ does not change either $\norm{f}_{A(G)}$ or $\abs{f}$, so we
may assume that the largest coefficient is the positive constant $x_1$. Put
$\chi=\gamma_2\overline{\gamma_1}$. Since $\chi$ is a non-trivial character and $G$
is connected, $\chi(G)$ is a non-trivial connected compact subgroup of $\mathbb T$;
hence $\chi(G)=\mathbb T$. Consequently, there exists $x_0\in G$ for which the
contribution of the second largest coefficient points exactly in the direction opposite
to the first one. Evaluating at $x_0$ and estimating the remaining Fourier series by its
$\ell^1$ norm, we obtain
\[
  \delta
  \le\abs{f(x_0)}
  \le x_1-x_2+\sum_{j\ge3}x_j
  \le x_1-x_2+(1-x_1-x_2)
  =1-2x_2.
\]
Thus
\begin{equation}
\label{eq:x2-bound}
  x_2\le b:=\frac{1-\delta}{2}.
\end{equation}
Combining \eqref{eq:two-largest} and \eqref{eq:x2-bound}, we also get
\begin{equation}
\label{eq:x1-linear}
  x_1\ge\delta-b=\frac{3\delta-1}{2}.
\end{equation}

If $x_1\ge\delta$, then $x_1\ge\Phi(\delta)$, because
$\sqrt{2\delta-1}\le\delta$. Hence assume that $x_1<\delta$. In this case
\[
  2b=1-\delta<1-x_1,
\]
while \eqref{eq:x1-linear} gives
\[
  1-x_1\le\frac{3(1-\delta)}{2}=3b.
\]
Every term of the Fourier tail $(x_j)_{j\ge2}$ is bounded by $b$, by
\eqref{eq:x2-bound}, and its total mass is at most $1-x_1$. Lemma~\ref{lem:packing}
therefore yields
\[
  \sum_{j\ge2}x_j^2
  \le 2b^2+(1-x_1-2b)^2
  =2b^2+(\delta-x_1)^2.
\]
Using Parseval once more,
\[
  \delta^2
  \le x_1^2+2b^2+(\delta-x_1)^2.
\]
After substituting $b=(1-\delta)/2$ and simplifying, this becomes
\[
  0\le x_1^2-\delta x_1+\frac{(1-\delta)^2}{4}.
\]
The roots of the quadratic polynomial on the right are
\[
  \frac{\delta-\sqrt{2\delta-1}}{2}
  \quad\text{and}\quad
  \frac{\delta+\sqrt{2\delta-1}}{2}.
\]
On the other hand, \eqref{eq:x1-linear} implies
\[
  x_1\ge\frac{3\delta-1}{2}
  >\frac{\delta-\sqrt{2\delta-1}}{2}.
\]
Since the quadratic is non-negative at $x_1$, we must therefore have
\[
  x_1\ge\frac{\delta+\sqrt{2\delta-1}}{2}=\Phi(\delta),
\]
as required.
\end{proof}

\begin{remark}
\label{rem:comparison-largest}
For every $1/2<\delta<1$ one has
\[
  \Phi(\delta)>\Psi(\delta).
\]
Indeed, after two elementary squarings the desired inequality reduces to
\[
  (3\delta-1)^2-(\delta+1)^2(2\delta-1)
  =2(1-\delta)^3>0.
\]
Thus Theorem~\ref{thm:improved-largest-coefficient} strictly improves the estimate
from Theorem~\ref{thm:OW-largest} throughout the non-trivial range. Notice also that
\[
  \Phi(\delta)>\frac12
  \quad\Longleftrightarrow\quad
  \delta>2-\sqrt2.
\]
\end{remark}

We shall also use the following elementary inversion lemma of Nikolski. Recall that a
unital Banach algebra $\mathcal A$ is said to split at the unit if
\[
  \mathcal A=\C e\oplus\mathcal A_0
  \qquad\text{and}\qquad
  \norm{\lambda e+a}_{\mathcal A}
  =\abs{\lambda}+\norm{a}_{\mathcal A}
  \quad
  (\lambda\in\C,\ a\in\mathcal A_0).
\]

\begin{lemma}[{\cite[Lemma~1.4.3]{Nikolski1999}}]
\label{lem:nikolski-splitting}
Let $\mathcal A$ be a unital Banach algebra splitting at the unit, and let
$f=\lambda e+f_0$ with $f_0\in\mathcal A_0$. If
\[
  \frac12<\eta\le\abs{\lambda}\le\norm{f}_{\mathcal A}\le1,
\]
then $f$ is invertible in $\mathcal A$ and
\[
  \norm{f^{-1}}_{\mathcal A}\le\frac{1}{2\eta-1}.
\]
In particular, taking $\eta=\abs{\lambda}$ gives
\[
  \norm{f^{-1}}_{\mathcal A}\le\frac{1}{2\abs{\lambda}-1}.
\]
\end{lemma}

\begin{theorem}
\label{thm:inverse-bound-connected}
Let $G$ be a compact connected Abelian group, let
\[
  \delta>2-\sqrt2,
\]
and let $f\in A(G)$ satisfy
\[
  \norm{f}_{A(G)}\le1,
  \qquad
  \abs{f(x)}\ge\delta
  \quad (x\in G).
\]
Then $f$ is invertible in $A(G)$ and
\[
  \norm{f^{-1}}_{A(G)}
  \le
  \frac{1}{\delta+\sqrt{2\delta-1}-1}.
\]
\end{theorem}

\begin{proof}
The assumptions imply $\delta\le\norm f_\infty\le\norm f_{A(G)}\le1$. Hence
Theorem~\ref{thm:improved-largest-coefficient} applies, and we may choose
$\gamma_0\in\widehat G$
such that
\[
  a:=\abs{\fhat(\gamma_0)}
  =\norm{\fhat}_{\infty}
  \ge\Phi(\delta)>\frac12.
\]
Replacing $f$ by $c\,\overline{\gamma_0}f$ for a suitable $c\in\mathbb T$,
we may assume that $\fhat(1)=a$. Then
\[
  A(G)=\C\one\oplus A_0,
  \qquad
  A_0:=\bigl\{h\in A(G):\widehat h(1)=0\bigr\},
\]
and this decomposition splits the $A(G)$-norm. Hence
Lemma~\ref{lem:nikolski-splitting}, applied with $\eta=a$, gives
\[
  \norm{f^{-1}}_{A(G)}
  \le\frac{1}{2a-1}
  \le\frac{1}{2\Phi(\delta)-1}
  =\frac{1}{\delta+\sqrt{2\delta-1}-1}.
\]
Since multiplication by a character and by a unimodular scalar preserves the
$A(G)$-norm, it also preserves the norm of the inverse. This proves the claim.
\end{proof}

\begin{example}
\label{ex:no-large-fourier-coefficient}
The preceding argument cannot cover the whole range $\delta>1/2$ by forcing a
single Fourier coefficient to have modulus greater than $1/2$. Indeed, for
$t\in\R$ consider the Blaschke factor evaluated on the unit circle,
\[
  B(e^{it}):=\frac{e^{it}-\frac25}{1-\frac25 e^{it}},
\]
and set
\[
  f(t):=\frac59 B(e^{it}).
\]
Since $\abs{B(e^{it})}=1$, we have
\[
  \abs{f(t)}=\frac59>\frac12
  \qquad (t\in\R).
\]
Moreover,
\[
  B(e^{it})
  =-\frac25
   +\frac{21}{25}\sum_{n=1}^{\infty}
      \left(\frac25\right)^{n-1}e^{int},
\]
so that
\[
  f(t)
  =-\frac29
   +\frac{7}{15}\sum_{n=1}^{\infty}
      \left(\frac25\right)^{n-1}e^{int}.
\]
Consequently,
\[
  \norm{f}_{A(\T)}
  =\frac29
   +\frac{7}{15}\sum_{n=1}^{\infty}
      \left(\frac25\right)^{n-1}
  =\frac29+\frac79
  =1,
\]
whereas
\[
  \norm{\fhat}_{\infty}
  =\frac{7}{15}
  <\frac12.
\]
Thus the standard assumptions $\norm{f}_{A(\T)}\le1$ and
$\inf_{t\in\R}\abs{f(t)}>1/2$ do not imply that any Fourier coefficient has modulus
larger than $1/2$. In particular, a proof of norm-controlled inversion throughout the
full range $\delta>1/2$ must use information beyond the size of a single Fourier
coefficient.
\end{example}

\subsection{Fourier algebras}
\label{subsec:fourier-algebra-positive}

We now pass from the connected case to arbitrary compact Abelian groups. The
bound obtained below is non-effective, but it is uniform both in the function and in
the underlying group.

\begin{theorem}[Uniform norm-controlled inversion above $1/2$]
\label{thm:uniform-fourier-inversion}
For every $\delta>1/2$ there exists a constant $C_A(\delta)<\infty$ with the
following property. For every compact Abelian group $G$ and every $f\in A(G)$
satisfying
\begin{equation}
  \norm{f}_{A(G)}\le1,
  \qquad
  \inf_{x\in G}\abs{f(x)}\ge\delta,
  \label{eq:uniform-fourier-assumptions}
\end{equation}
the function $f$ is invertible in $A(G)$ and
\begin{equation}
  \norm{f^{-1}}_{A(G)}\le C_A(\delta).
  \label{eq:uniform-fourier-conclusion}
\end{equation}
The constant $C_A(\delta)$ is independent of $G$.
\end{theorem}

\paragraph{Roadmap of the proof.}
We argue by contradiction. The proof proceeds in the following steps.
\begin{enumerate}[label=\textup{(\arabic*)}]
\item A failure of uniformity, even with the compact group allowed to vary, is first
reduced to a bad sequence $(f_j)$ in one compact Abelian group $G$ with countable dual.
This turns the problem into a sequential compactness question for Fourier coefficients.
\item We extract all non-vanishing concentration profiles of $(f_j)$ up to modulation.
This gives
\[
  f_j=h_j+w_j,
\]
where $h_j$ is the sum of pairwise asymptotically separated profiles, while the largest
Fourier coefficient of $w_j$ tends to zero. The decomposition has exact asymptotic
balances for both the Wiener norm and the $L^2$ energy.
\item The profiles are assembled into a single limiting core $H$ on a compact product
group $X$. The frequency relations of the approximating cores define compact subgroups
$K_j\subset X$ and a limiting subgroup $K$. The original lower bound $|f_j|\ge\delta$
passes to the limit and yields $|H|\ge\delta$ on $K$.
\item If $P$ is the total Wiener mass of the profiles and $\beta$ is the limiting Wiener
mass of the remainder, the mass and energy balances imply
\[
  P+\beta\le1,
  \qquad
  P\ge\delta.
\]
Hence $\beta\le1-\delta<\delta$. This is the decisive mass gap and the only place where
the strict threshold $\delta>1/2$ is used.
\item Choose $\rho$ with $\beta<\rho<\delta$. On $K\times\T$ the limiting resolvent
$H(y)-\rho z$ is invertible. Its inverse has only non-negative $z$-frequencies. A
stability-under-restriction argument transports this one-sided inverse to $K_j\times\T$
and then back to $G\times\T$, giving a uniform bound for $(h_j-\rho z)^{-1}$.
\item The non-negative $z$-spectrum converts that resolvent bound into simultaneous
control of all inverse powers of $h_j$. Since $\|w_j\|_{A(G)}\to\beta<\rho$, one
may choose $r$ with $\beta<r<\rho$ and arrange $\|w_j\|_{A(G)}\le r$ for all
remaining $j$. The series
\[
  \sum_{k\ge0}(-1)^k h_j^{-k-1}w_j^k
\]
then has tails bounded uniformly by a geometric factor $(r/\rho)^N$. It equals
$f_j^{-1}$, contradicting the assumed blow-up of the inverse norms.
\end{enumerate}
This organization is the point at which the proof departs from the earlier
largest-coefficient strategy: the entire family of profiles, rather than one dominant
Fourier coefficient, carries the information needed for inversion. Throughout the
contradiction arguments below, after passing to a subsequence we freely discard finitely
many initial indices and relabel the remaining sequence; all asymptotic statements are
understood in this sense.

The qualitative invertibility in the statement is the classical Wiener lemma for compact
Abelian groups. More precisely, it follows from the description of the maximal ideal space
of $\ell^1(\widehat G)$ together with Pontryagin duality and the Gelfand invertibility
criterion; see \cite[Theorems~1.2.2 and~1.7.2, Appendix~D4(c)]{Rudin1990}. The content of
Theorem~\ref{thm:uniform-fourier-inversion} is the uniform control of the norm of the
inverse. Notice also that the number $1/2$ is a universal threshold for the class of all
compact Abelian groups.

\subsubsection{Reduction to a fixed group and to a countable dual}

Suppose that Theorem~\ref{thm:uniform-fourier-inversion} fails for some fixed
$\delta>1/2$. Then for every $j\ge1$ there are a compact Abelian group $G_j$ and a
function $f_j\in A(G_j)$ such that
\begin{equation}
  \norm{f_j}_{A(G_j)}\le1,
  \qquad
  \abs{f_j}\ge\delta,
  \qquad
  \norm{f_j^{-1}}_{A(G_j)}\ge j.
  \label{eq:varying-bad-sequence}
\end{equation}
Let
\[
  G:=\prod_{j=1}^{\infty}G_j
\]
and let $\pi_j:G\to G_j$ be the $j$th coordinate projection. The pullback
$f\mapsto f\circ\pi_j$ is an isometric embedding $A(G_j)\hookrightarrow A(G)$,
because
\[
  \widehat G\cong\bigoplus_{j=1}^{\infty}\widehat{G_j}.
\]
Thus, for $F_j:=f_j\circ\pi_j$,
\[
  \norm{F_j}_{A(G)}=\norm{f_j}_{A(G_j)},
  \qquad
  \inf_G\abs{F_j}=\inf_{G_j}\abs{f_j},
\]
and, since $F_j^{-1}=f_j^{-1}\circ\pi_j$,
\[
  \norm{F_j^{-1}}_{A(G)}=\norm{f_j^{-1}}_{A(G_j)}.
\]
Consequently it is enough to exclude a bad sequence in one fixed compact Abelian
group.

\begin{proposition}
\label{prop:fixed-group-no-bad-sequence}
Let $G$ be a fixed compact Abelian group and let $\delta>1/2$. There is no sequence
$(f_j)\subset A(G)$ such that
\begin{equation}
  \norm{f_j}_{A(G)}\le1,
  \qquad
  \abs{f_j(x)}\ge\delta\quad(x\in G),
  \qquad
  \norm{f_j^{-1}}_{A(G)}\longrightarrow\infty.
  \label{eq:bad-fixed-group}
\end{equation}
\end{proposition}

For the proof of Proposition~\ref{prop:fixed-group-no-bad-sequence} we may further
reduce to a group with countable dual.

\begin{lemma}[Reduction to a countable dual]
\label{lem:countable-dual-reduction}
If a compact Abelian group $G$ admits a sequence satisfying
\eqref{eq:bad-fixed-group}, then such a sequence exists on a compact Abelian group
$G_0$ whose dual $\widehat{G_0}$ is countable.
\end{lemma}

\begin{proof}
Put $\Gamma=\widehat G$. For every $j$, the Fourier support
$\supp\fhat_j$ is at most countable. Hence the subgroup
\[
  \Gamma_0:=
  \left\langle\bigcup_{j=1}^{\infty}\supp\fhat_j\right\rangle
  \subset\Gamma
\]
is countable. Let
\[
  N:=\Gamma_0^{\perp}
  =\{x\in G:\gamma(x)=1\text{ for all }\gamma\in\Gamma_0\}.
\]
Since $\Gamma$ is discrete, $\Gamma_0$ is closed, and the duality between closed
subgroups and quotient groups \cite[Theorem~2.1.2]{Rudin1990} yields
\[
  \widehat{G/N}\cong N^{\perp}=\Gamma_0.
\]
Each $f_j$ is $N$-invariant, so there is $F_j\in A(G/N)$ such that
$f_j=F_j\circ q$, where $q:G\to G/N$ is the quotient map. The identification of
the dual groups gives
\[
  \norm{F_j}_{A(G/N)}=\norm{f_j}_{A(G)},
  \qquad
  \inf_{G/N}\abs{F_j}=\inf_G\abs{f_j}.
\]
By the Wiener lemma, $F_j^{-1}\in A(G/N)$, and uniqueness of the inverse gives
$f_j^{-1}=F_j^{-1}\circ q$. Hence
\[
  \norm{F_j^{-1}}_{A(G/N)}=\norm{f_j^{-1}}_{A(G)}.
\]
Taking $G_0=G/N$ proves the assertion.
\end{proof}

If $\widehat{G_0}$ is finite, then $A(G_0)$ is finite-dimensional. The set
\[
  \left\{f\in A(G_0):
  \norm f_{A(G_0)}\le1,
  \ \inf_{G_0}\abs f\ge\delta\right\}
\]
is compact and inversion is continuous on it. Hence the inverse norms are uniformly
bounded. Therefore, from now on, in the proof of
Proposition~\ref{prop:fixed-group-no-bad-sequence} we may assume that
\begin{equation}
  \Gamma:=\widehat G
  \quad\text{is a countably infinite discrete Abelian group.}
  \label{eq:countably-infinite-dual}
\end{equation}

We use additive notation in the discrete dual group $\Gamma$; thus
$(\gamma-\eta)(x)=\gamma(x)\overline{\eta(x)}$. For $\sigma\in\Gamma$ define the modulation
\[
  M_{\sigma}f(x):=\sigma(x)f(x).
\]
Then
\begin{equation}
  \widehat{M_{\sigma}f}(\gamma)=\fhat(\gamma-\sigma),
  \qquad
  \norm{M_{\sigma}f}_{A(G)}=\norm f_{A(G)},
  \qquad
  \norm{M_{\sigma}f}_2=\norm f_2.
  \label{eq:modulation}
\end{equation}
Parseval's identity gives
\begin{equation}
  \norm f_2^2
  =\sum_{\gamma\in\Gamma}\abs{\fhat(\gamma)}^2
  \le \norm{\fhat}_{\infty}\norm f_{A(G)}.
  \label{eq:L2-by-fourier-sup}
\end{equation}
For a sequence $(\sigma_j)\subset\Gamma$ we write $\sigma_j\to\infty$ if it
eventually leaves every finite subset of $\Gamma$. We shall repeatedly use the Riemann--Lebesgue property
\cite[Theorem~1.2.4(a)]{Rudin1990}:
\begin{equation}
  \widehat u(\gamma)\longrightarrow0
  \qquad(\gamma\to\infty)
  \label{eq:Riemann-Lebesgue-discrete}
\end{equation}
for $u\in L^1(G)$; in particular it holds for $u\in A(G)$.

\subsubsection{Profile decomposition}

We say that $u_j\to u$ coefficientwise if
\[
  \widehat u_j(\gamma)\longrightarrow\widehat u(\gamma)
  \qquad(\gamma\in\Gamma).
\]
The following elementary splitting property of the Wiener norm is the basic
compactness device.

\begin{lemma}[Splitting of the Wiener norm]
\label{lem:wiener-norm-splitting}
Let $(u_j)$ be bounded in $A(G)$ and suppose that $u_j\to u$ coefficientwise.
Then $u\in A(G)$ and
\begin{equation}
  \norm{u_j}_{A(G)}-\norm{u_j-u}_{A(G)}
  \longrightarrow\norm u_{A(G)}.
  \label{eq:wiener-norm-splitting}
\end{equation}
\end{lemma}

\begin{proof}
Fatou's lemma gives $u\in A(G)$. Moreover, for every $\gamma\in\Gamma$,
\[
  \abs{\widehat u_j(\gamma)}
  -\abs{\widehat u_j(\gamma)-\widehat u(\gamma)}
  \longrightarrow\abs{\widehat u(\gamma)},
\]
while the reverse triangle inequality yields the domination
\[
  \left|
  \abs{\widehat u_j(\gamma)}
  -\abs{\widehat u_j(\gamma)-\widehat u(\gamma)}
  \right|
  \le\abs{\widehat u(\gamma)}.
\]
Summing and applying dominated convergence proves
\eqref{eq:wiener-norm-splitting}.
\end{proof}

\begin{theorem}[Profile decomposition]
\label{thm:profile-decomposition}
Let $(f_j)$ be bounded in $A(G)$ and assume that $\Gamma=\widehat G$ is
countable. After passing to a subsequence, the limits
\[
  L:=\lim_j\norm{f_j}_{A(G)},
  \qquad
  E:=\lim_j\norm{f_j}_2^2
\]
exist, and there exist a finite or countable index set $I$, functions $G_r\in A(G)$,
frequencies $\gamma_j^{(r)}\in\Gamma$, and a decomposition
\begin{equation}
  f_j=h_j+w_j,
  \qquad
  h_j:=\sum_{r\in I}M_{\gamma_j^{(r)}}G_r,
  \label{eq:profile-decomposition}
\end{equation}
with the following properties:
\begin{enumerate}[label=\textup{(\roman*)}]
\item if $r\ne s$, then
\[
  \gamma_j^{(r)}-\gamma_j^{(s)}\longrightarrow\infty;
\]
\item the series defining $h_j$ converges absolutely in $A(G)$ for every $j$;
\item for some $\beta\ge0$,
\begin{equation}
  \norm{\widehat{w_j}}_{\infty}\longrightarrow0,
  \qquad
  \norm{w_j}_{A(G)}\longrightarrow\beta;
  \label{eq:remainder-properties}
\end{equation}
\item one has
\begin{equation}
  L=\sum_{r\in I}\norm{G_r}_{A(G)}+\beta;
  \label{eq:profile-mass-balance}
\end{equation}
\item one has
\begin{equation}
  E=\sum_{r\in I}\norm{G_r}_2^2.
  \label{eq:profile-energy-balance}
\end{equation}
\end{enumerate}
In particular,
\begin{equation}
  \norm{w_j}_2\longrightarrow0.
  \label{eq:remainder-L2}
\end{equation}
\end{theorem}

\begin{proof}
Passing first to a subsequence, assume that both
$L=\lim_j\norm{f_j}_{A(G)}$ and $E=\lim_j\norm{f_j}_2^2$ exist. Put
$R_j^{(0)}=f_j$ and, after a further subsequence, assume that
\[
  \lambda_0:=\lim_j\norm{\widehat{R_j^{(0)}}}_{\infty}
\]
exists. If $\lambda_0=0$, there are no non-zero profiles: take $h_j=0$ and
$w_j=f_j$. Suppose therefore that $\lambda_0>0$.

Choose $\gamma_j^{(1)}\in\Gamma$ so that
\[
  \abs{\widehat{R_j^{(0)}}(\gamma_j^{(1)})}
  \ge\frac12\norm{\widehat{R_j^{(0)}}}_{\infty},
\]
and put $U_j^{(1)}=M_{-\gamma_j^{(1)}}R_j^{(0)}$. Since $\Gamma$ is
countable, a diagonal argument yields a subsequence on which every Fourier
coefficient of $U_j^{(1)}$ converges. Fatou's lemma shows that the limiting
coefficients define a function $G_1\in A(G)$, with
\[
  \abs{\widehat G_1(0)}\ge\lambda_0/2>0.
\]
Set
\[
  R_j^{(1)}:=R_j^{(0)}-M_{\gamma_j^{(1)}}G_1.
\]
Then $M_{-\gamma_j^{(1)}}R_j^{(1)}\to0$ coefficientwise and
Lemma~\ref{lem:wiener-norm-splitting} gives
\[
  \norm{R_j^{(0)}}_{A(G)}-\norm{R_j^{(1)}}_{A(G)}
  \longrightarrow\norm{G_1}_{A(G)}.
\]

Continue inductively. After extracting $R$ profiles we have
\[
  R_j^{(R)}
  =f_j-\sum_{r=1}^{R}M_{\gamma_j^{(r)}}G_r,
\]
the centers are pairwise divergent, every already extracted frame sees the remainder
coefficientwise tending to zero, and
\begin{equation}
  \norm{f_j}_{A(G)}-\norm{R_j^{(R)}}_{A(G)}
  \longrightarrow\sum_{r=1}^{R}\norm{G_r}_{A(G)}.
  \label{eq:finite-profile-mass}
\end{equation}
After passing to a subsequence, let
$\lambda_R=\lim_j\norm{\widehat{R_j^{(R)}}}_{\infty}$. If $\lambda_R=0$, the procedure stops.
Otherwise choose $\gamma_j^{(R+1)}$ so that
\[
  \abs{\widehat{R_j^{(R)}}(\gamma_j^{(R+1)})}
  \ge\frac12\norm{\widehat{R_j^{(R)}}}_{\infty},
\]
and extract a coefficientwise limit $G_{R+1}$ after demodulation. Then
\begin{equation}
  \abs{\widehat G_{R+1}(0)}\ge\lambda_R/2.
  \label{eq:new-profile-detection}
\end{equation}
The new center must diverge from every preceding one. Indeed, if
$\gamma_j^{(R+1)}-\gamma_j^{(s)}$ did not escape every finite subset of the
discrete group $\Gamma$, then some finite set would contain infinitely many of these
differences. One value $d\in\Gamma$ would therefore occur along an infinite
subsequence. Along that subsequence,
\[
  \widehat{R_j^{(R)}}(\gamma_j^{(R+1)})
  =\widehat{M_{-\gamma_j^{(s)}}R_j^{(R)}}(d)\longrightarrow0,
\]
contradicting $\lambda_R>0$. Subtracting the new profile preserves the local
coefficientwise vanishing around the old centers, because
\eqref{eq:Riemann-Lebesgue-discrete} kills the translated coefficients of
$G_{R+1}$. Lemma~\ref{lem:wiener-norm-splitting} gives the next instance of
\eqref{eq:finite-profile-mass}.

The subsequences chosen during the induction are nested. We take a diagonal
subsequence whose tail, for each fixed $R$, lies in the subsequence selected at the
$R$-th stage, and then relabel it by $j$.
Extend each sequence of centers $(\gamma_j^{(r)})_j$ arbitrarily over the finitely many
indices preceding the extraction of the $r$-th profile. Thus a single subsequence is
obtained for which all finite levels of the construction are valid. For every fixed $R$,
\begin{equation}
  \lim_j\norm{R_j^{(R)}}_{A(G)}
  =L-\sum_{r=1}^{R}\norm{G_r}_{A(G)},
  \label{eq:fixed-R-remainder-mass}
\end{equation}
so
\begin{equation}
  \sum_r\norm{G_r}_{A(G)}\le L.
  \label{eq:profile-summability}
\end{equation}
Moreover, \eqref{eq:new-profile-detection} implies
$\lambda_R\le2\norm{G_{R+1}}_{A(G)}$ whenever the construction continues; hence
$\lambda_R\to0$ if infinitely many profiles occur.

In the finite case, the terminal remainder is $w_j$. In the infinite case, choose
integers $N_R$ strictly increasing, with $N_R\ge R$, so that for $j\ge N_R$ both
\[
  \abs{\norm{\widehat{R_j^{(R)}}}_{\infty}-\lambda_R}\le R^{-1}
\]
and
\[
  \left|
  \norm{R_j^{(R)}}_{A(G)}-
  \left(L-\sum_{r=1}^{R}\norm{G_r}_{A(G)}\right)
  \right|\le R^{-1}.
\]
For $j<N_1$ put $R(j)=0$, and for $j\ge N_1$ let
$R(j)=\max\{R:N_R\le j\}$. These finitely many initial indices do not affect any
asymptotic conclusion. Define
\[
  h_j:=\sum_{r=1}^{\infty}M_{\gamma_j^{(r)}}G_r,
  \qquad
  w_j:=f_j-h_j.
\]
The series converges absolutely by \eqref{eq:profile-summability}, and
\[
  w_j=R_j^{(R(j))}
  -\sum_{r>R(j)}M_{\gamma_j^{(r)}}G_r.
\]
The two diagonal-tail estimates are explicit. First,
\begin{align*}
  \norm{\widehat{w_j}}_{\infty}
  &\le
  \norm{\widehat{R_j^{(R(j))}}}_{\infty}
  +
  \sum_{r>R(j)}\norm{G_r}_{A(G)}
  \\
  &\le
  \lambda_{R(j)}+\frac1{R(j)}
  +
  \sum_{r>R(j)}\norm{G_r}_{A(G)}
  \longrightarrow0.
\end{align*}
Second,
\[
  \left|
  \norm{w_j}_{A(G)}
  -
  \norm{R_j^{(R(j))}}_{A(G)}
  \right|
  \le
  \sum_{r>R(j)}\norm{G_r}_{A(G)}
  \longrightarrow0.
\]
By the choice of $N_R$,
\[
  \norm{R_j^{(R(j))}}_{A(G)}
  -
  \left(
    L-\sum_{r=1}^{R(j)}\norm{G_r}_{A(G)}
  \right)
  \longrightarrow0,
\]
and therefore
\[
  \norm{w_j}_{A(G)}\longrightarrow
  \beta:=L-\sum_r\norm{G_r}_{A(G)}.
\]
Then \eqref{eq:L2-by-fourier-sup} gives $\norm{w_j}_2\to0$.

Finally, if $r\ne s$, then
\[
  \left\langle
  M_{\gamma_j^{(r)}}G_r,
  M_{\gamma_j^{(s)}}G_s
  \right\rangle
  =\int_G
  (\gamma_j^{(r)}-\gamma_j^{(s)})(x)
  G_r(x)\overline{G_s(x)}\,dm_G(x),
\]
which tends to zero by \eqref{eq:Riemann-Lebesgue-discrete}. Thus every finite
collection of profiles is asymptotically orthogonal in $L^2$. The profile tail is
uniformly small because
\[
  \left\|
  \sum_{r>R}M_{\gamma_j^{(r)}}G_r
  \right\|_2
  \le\sum_{r>R}\norm{G_r}_2
  \le\sum_{r>R}\norm{G_r}_{A(G)}\longrightarrow0.
\]
Hence
\[
  \norm{h_j}_2^2\longrightarrow\sum_r\norm{G_r}_2^2,
\]
and, since $f_j=h_j+w_j$ with $\norm{w_j}_2\to0$, the energy identity
\eqref{eq:profile-energy-balance} follows.
\end{proof}

\subsubsection{The limiting profile core}

Apply Theorem~\ref{thm:profile-decomposition} to a hypothetical sequence satisfying
\eqref{eq:bad-fixed-group}. Passing to a subsequence, assume that
\[
  L:=\lim_j\norm{f_j}_{A(G)}\le1,
  \qquad
  E:=\lim_j\norm{f_j}_2^2
\]
exist. Since $\abs{f_j}\ge\delta$,
\begin{equation}
  E\ge\delta^2.
  \label{eq:energy-lower-bound}
\end{equation}
Let $I$ be the profile index set. Define
\[
  X:=G\times\T^I
\]
and
\begin{equation}
  H\bigl(x,(z_r)_{r\in I}\bigr)
  :=\sum_{r\in I}z_rG_r(x).
  \label{eq:limiting-core-H}
\end{equation}
The series converges absolutely in $A(X)$ by
\eqref{eq:profile-summability}. For every $j$ define
\[
  \phi_j:G\to X,
  \qquad
  \phi_j(x)
  :=\bigl(x,(\gamma_j^{(r)}(x))_{r\in I}\bigr),
\]
and let $K_j:=\phi_j(G)$. Then $K_j$ is a compact subgroup of $X$,
$\phi_j$ is an isomorphism of $G$ onto $K_j$, and
\begin{equation}
  H\circ\phi_j=h_j.
  \label{eq:H-pullback}
\end{equation}

The dual group
\[
  \widehat X=\Gamma\oplus\Z^{(I)}
\]
is countable. For $\lambda=(\xi,(a_r))\in\widehat X$,
\[
  \lambda\circ\phi_j
  =\xi+\sum_{r\in I}a_r\gamma_j^{(r)}\in\Gamma,
\]
where only finitely many $a_r$ are non-zero. Passing to one more subsequence, we may
assume that, for every $\lambda\in\widehat X$, either
$\lambda|_{K_j}=1$ for all sufficiently large $j$, or
$\lambda|_{K_j}\ne1$ for all sufficiently large $j$. Put
\[
  \mathcal L
  :=\{\lambda\in\widehat X:
  \lambda|_{K_j}=1\text{ for all sufficiently large }j\}
\]
and
\begin{equation}
  K:=\mathcal L^{\perp}\subset X.
  \label{eq:limiting-subgroup-K}
\end{equation}
Then $\mathcal L$ is a subgroup of $\widehat X$ and, by the double annihilator
theorem \cite[Lemma~2.1.3]{Rudin1990},
\[
  K^{\perp}=\mathcal L.
\]

\begin{lemma}[Convergence of Haar measures]
\label{lem:Haar-measure-convergence}
Let $m_{K_j}$ and $m_K$ denote normalized Haar measures. Then
\[
  m_{K_j}\stackrel{*}{\longrightarrow}m_K.
\]
\end{lemma}

\begin{proof}
For $\lambda\in\widehat X$,
\[
  \int_{K_j}\lambda\,dm_{K_j}
  =
  \begin{cases}
    1,&\lambda|_{K_j}=1,\\
    0,&\lambda|_{K_j}\ne1.
  \end{cases}
\]
By construction this converges to $1$ exactly for
$\lambda\in\mathcal L=K^{\perp}$, which is precisely
$\int_K\lambda\,dm_K$. The convergence therefore holds for trigonometric
polynomials and, by Stone--Weierstrass, for every continuous function on $X$.
\end{proof}

\begin{lemma}[Lower bound on the limiting core]
\label{lem:H-lower-bound}
For every $y\in K$,
\[
  \abs{H(y)}\ge\delta.
\]
\end{lemma}

\begin{proof}
Since $f_j=h_j+w_j$ and $\abs{f_j}\ge\delta$,
\[
  (\delta-\abs{h_j})_+\le\abs{w_j}.
\]
The pushforward of $m_G$ under $\phi_j$ is $m_{K_j}$, so
\[
  \int_{K_j}(\delta-\abs H)_+^2\,dm_{K_j}
  =\int_G(\delta-\abs{h_j})_+^2\,dm_G
  \le\norm{w_j}_2^2\longrightarrow0.
\]
Lemma~\ref{lem:Haar-measure-convergence} gives
\[
  \int_K(\delta-\abs H)_+^2\,dm_K=0.
\]
Haar measure has full support on $K$, while the integrand is continuous and
non-negative. Hence it vanishes identically.
\end{proof}

\subsubsection{Stability under restriction}

We next need a device that transfers invertibility from the limiting subgroup $K$ back
to the approximating subgroups $K_j$. The key point is that the transfer must preserve
one-sided Fourier spectrum in an additional circle variable.

\begin{lemma}[Stability of restriction norms]
\label{lem:restriction-norm-stability}
Let $Y$ be a compact Abelian group and let $M_j,M\subset Y$ be compact subgroups
such that, for every $\lambda\in\widehat Y$, the relation $\lambda|_{M_j}=1$
eventually stabilizes to $\lambda|_M=1$. Then, for every $F\in A(Y)$,
\[
  \norm{F|_{M_j}}_{A(M_j)}\longrightarrow\norm{F|_M}_{A(M)}.
\]
The same conclusion holds for restriction from $Y\times\T$ to $M_j\times\T$ and
$M\times\T$. In particular, restricting an element with only non-negative Fourier
frequencies in the $\T$-variable preserves that one-sided spectrum.
\end{lemma}

\begin{proof}
Let $P=\sum_{\lambda\in E}a_\lambda\lambda$ be a trigonometric polynomial. Its
restriction norm is obtained by grouping together those frequencies that have the
same restriction. For $\lambda,\mu\in E$, this is determined by the finitely many
relations $(\lambda-\mu)|_{M_j}=1$, all of which eventually agree with the
corresponding relations for $M$. Hence, for all sufficiently large $j$,
\[
  \norm{P|_{M_j}}_{A(M_j)}=\norm{P|_M}_{A(M)}.
\]
Given $F\in A(Y)$ and $\varepsilon>0$, choose such a polynomial $P$ with
$\norm{F-P}_{A(Y)}<\varepsilon$. Since restriction is contractive, for all large
$j$,
\[
  \left|\norm{F|_{M_j}}_{A(M_j)}-\norm{F|_M}_{A(M)}\right|\le2\varepsilon.
\]
This proves the asserted convergence. For $Y\times\T$, restriction does not identify
different integer frequencies in the second coordinate, so the same proof applies and
preserves non-negative $\T$-frequencies.
\end{proof}

For a compact subgroup $M\subset Y$, the restriction map
\[
  A(Y)\longrightarrow A(M)
\]
is a contraction and a metric quotient. Indeed, by \cite[Theorem~2.1.2]{Rudin1990},
$\widehat M\cong\widehat Y/M^{\perp}$, and one obtains a norm-preserving lift by
choosing one representative from each coset of $M^{\perp}$ carrying a non-zero
Fourier coefficient.

For an additional variable $z\in\T$, put
\[
  A_+(Y\times\T)
  :=\left\{F\in A(Y\times\T):
  \widehat F(\lambda,n)=0\text{ whenever }n<0\right\}.
\]
This is a closed unital subalgebra of $A(Y\times\T)$. Since
\[
  (M\times\T)^{\perp}=M^{\perp}\times\{0\},
\]
passing to the quotient changes only the $\widehat Y$-coordinate and never the
integer degree in $z$. Consequently, every $u\in A_+(M\times\T)$ admits a lift
$U\in A_+(Y\times\T)$ satisfying
\begin{equation}
  U|_{M\times\T}=u,
  \qquad
  \norm U_{A(Y\times\T)}=\norm u_{A(M\times\T)}.
  \label{eq:positive-spectrum-lift}
\end{equation}

\begin{lemma}[Analytic pullback]
\label{lem:analytic-pullback}
Let $Y$ be a compact Abelian group and let $M_j,M\subset Y$ satisfy the
stabilization assumption of Lemma~\ref{lem:restriction-norm-stability}. Let
$F\in A_+(Y\times\T)$. Assume that $F|_{M\times\T}$ is invertible and
\begin{equation}
  (F|_{M\times\T})^{-1}\in A_+(M\times\T).
  \label{eq:limit-inverse-positive-spectrum}
\end{equation}
Then, for all sufficiently large $j$, $F|_{M_j\times\T}$ is invertible,
\begin{equation}
  \sup_{j\ge j_0}
  \norm{(F|_{M_j\times\T})^{-1}}_{A(M_j\times\T)}<\infty,
  \label{eq:uniform-analytic-pullback}
\end{equation}
and
\begin{equation}
  (F|_{M_j\times\T})^{-1}\in A_+(M_j\times\T).
  \label{eq:positive-spectrum-pullback}
\end{equation}
\end{lemma}

\begin{proof}
Let
\[
  u:=(F|_{M\times\T})^{-1}\in A_+(M\times\T)
\]
and choose, using \eqref{eq:positive-spectrum-lift}, a norm-preserving lift
$Q\in A_+(Y\times\T)$. Set
\[
  R:=QF-1.
\]
Then $R\in A_+(Y\times\T)$ and $R|_{M\times\T}=0$. Since $R|_{M\times\T}=0$, Lemma~\ref{lem:restriction-norm-stability}
applied on $Y\times\T$ gives
\begin{equation}
  \norm{R|_{M_j\times\T}}_{A(M_j\times\T)}\longrightarrow0.
  \label{eq:R-restrictions-to-zero}
\end{equation}
For large $j$ put
$R_j=R|_{M_j\times\T}$ and assume $\norm{R_j}<1/2$. Then
\[
  (1+R_j)^{-1}=\sum_{k=0}^{\infty}(-R_j)^k
\]
converges in $A_+(M_j\times\T)$, and
\[
  (F|_{M_j\times\T})^{-1}
  =(1+R_j)^{-1}(Q|_{M_j\times\T}).
\]
This gives \eqref{eq:positive-spectrum-pullback} and, since restriction is
contractive,
\[
  \norm{(F|_{M_j\times\T})^{-1}}
  \le2\norm Q,
\]
which proves \eqref{eq:uniform-analytic-pullback}.
\end{proof}

The preservation of the one-sided $z$-spectrum is crucial. For a disconnected group,
one cannot argue that a continuous function avoiding the circle $\abs u=\rho$ lies
globally on one side of that circle. The spectral information retained in
Lemma~\ref{lem:analytic-pullback} replaces this topological argument.

\subsubsection{The mass gap and the resolvent}

Let
\[
  P:=\sum_{r\in I}\norm{G_r}_{A(G)}.
\]
The mass balance \eqref{eq:profile-mass-balance} gives
\[
  P+\beta=L\le1.
\]
On the other hand, Lemma~\ref{lem:H-lower-bound} and
$\norm H_{A(X)}=P$ give the direct comparison
\[
  \delta
  \le\norm{H|_K}_{\infty}
  \le\norm H_{A(X)}
  =P.
\]
Thus $P\ge\delta$, and hence
\begin{equation}
  \beta\le1-P\le1-\delta<\delta.
  \label{eq:key-mass-gap}
\end{equation}
This is the unique point at which the strict assumption $\delta>1/2$ enters the
argument. Choose
\begin{equation}
  \beta<\rho<\delta.
  \label{eq:rho-choice}
\end{equation}

On $X\times\T$ define
\begin{equation}
  \mathcal R_{\rho}(y,z):=H(y)-\rho z.
  \label{eq:limit-resolvent}
\end{equation}
This belongs to $A_+(X\times\T)$. On $K\times\T$, Lemma~\ref{lem:H-lower-bound}
gives
\[
  \abs{\mathcal R_{\rho}(y,z)}
  \ge\abs{H(y)}-\rho
  \ge\delta-\rho>0.
\]
The qualitative Wiener lemma therefore yields
\[
  U:=(\mathcal R_{\rho}|_{K\times\T})^{-1}
  \in A(K\times\T).
\]

\begin{lemma}[Positive $z$-spectrum of the limiting inverse]
\label{lem:limit-positive-spectrum}
One has
\[
  U\in A_+(K\times\T).
\]
\end{lemma}

\begin{proof}
Write the absolutely convergent Fourier series of $U$ grouped by the $z$-degree,
\[
  U(y,z)=\sum_{n\in\Z}U_n(y)z^n,
  \qquad
  U_n\in A(K),
  \qquad
  \sum_n\norm{U_n}_{A(K)}<\infty.
\]
For fixed $y\in K$ we have $\abs{H(y)}\ge\delta>\rho$, so
\[
  U(y,z)=\frac{1}{H(y)-\rho z}
\]
is holomorphic for $\abs z<\abs{H(y)}/\rho>1$. Its Fourier coefficients of
negative degree therefore vanish. By uniqueness of Fourier coefficients,
$U_n(y)=0$ for all $n<0$ and all $y\in K$, so $U_n=0$ for $n<0$.
\end{proof}

Apply Lemma~\ref{lem:analytic-pullback} with
\[
  Y=X,
  \qquad
  M_j=K_j,
  \qquad
  M=K,
  \qquad
  F=\mathcal R_{\rho}.
\]
For all sufficiently large $j$ we obtain
\begin{equation}
  \sup_j
  \norm{(\mathcal R_{\rho}|_{K_j\times\T})^{-1}}_{A(K_j\times\T)}<\infty
  \label{eq:uniform-resolvent-on-Kj}
\end{equation}
and
\begin{equation}
  (\mathcal R_{\rho}|_{K_j\times\T})^{-1}
  \in A_+(K_j\times\T).
  \label{eq:positive-resolvent-on-Kj}
\end{equation}
The group isomorphism
\[
  \phi_j\times\operatorname{id}_{\T}:
  G\times\T\longrightarrow K_j\times\T
\]
induces an isometry of Fourier algebras and pulls $\mathcal R_{\rho}$ back to
$(x,z)\mapsto h_j(x)-\rho z$. Hence
\begin{equation}
  V_j(x,z):=(h_j(x)-\rho z)^{-1}
  \in A_+(G\times\T)
  \label{eq:Vj-positive-spectrum}
\end{equation}
and
\begin{equation}
  \sup_j\norm{V_j}_{A(G\times\T)}<\infty.
  \label{eq:Vj-uniform-bound}
\end{equation}

\begin{lemma}
\label{lem:h-outside-rho}
For all sufficiently large $j$,
\[
  \inf_{x\in G}\abs{h_j(x)}>\rho.
\]
\end{lemma}

\begin{proof}
The invertibility of $h_j(x)-\rho z$ on $G\times\T$ first implies
$\abs{h_j(x)}\ne\rho$ for all $x$. If, for some $x_0\in G$,
$\abs{h_j(x_0)}<\rho$, then the scalar function of $z$ has the absolutely
convergent Laurent expansion
\[
  \frac{1}{h_j(x_0)-\rho z}
  =-\sum_{k=0}^{\infty}
  h_j(x_0)^k\rho^{-k-1}z^{-k-1}.
\]
In particular, the coefficient of $z^{-1}$ equals $-1/\rho$. This contradicts
\eqref{eq:Vj-positive-spectrum}, because evaluation at $x_0$ cannot create a negative
$z$-frequency. Thus $\abs{h_j(x)}>\rho$ everywhere. Since $G$ is compact and $h_j$ is
continuous, it follows in fact that $\min_{x\in G}\abs{h_j(x)}>\rho$.
\end{proof}

\begin{lemma}[Resolvent identity for inverse powers]
\label{lem:inverse-power-identity}
For all sufficiently large $j$,
\begin{equation}
  \norm{(h_j-\rho z)^{-1}}_{A(G\times\T)}
  =\sum_{k=0}^{\infty}
  \rho^k\norm{h_j^{-k-1}}_{A(G)}.
  \label{eq:inverse-power-identity}
\end{equation}
\end{lemma}

\begin{proof}
By Lemma~\ref{lem:h-outside-rho}, $\inf_G\abs{h_j}>\rho$. Since
$V_j\in A_+(G\times\T)$, write
\[
  V_j(x,z)=\sum_{k=0}^{\infty}A_{j,k}(x)z^k,
  \qquad
  \norm{V_j}_{A(G\times\T)}
  =\sum_{k=0}^{\infty}\norm{A_{j,k}}_{A(G)}.
\]
For fixed $x$, the geometric expansion gives
\[
  \frac{1}{h_j(x)-\rho z}
  =\sum_{k=0}^{\infty}\rho^kz^kh_j(x)^{-k-1},
  \qquad \abs z\le1.
\]
Uniqueness of the Fourier coefficients in $z$ yields
$A_{j,k}=\rho^kh_j^{-k-1}$, and summing the norms proves
\eqref{eq:inverse-power-identity}.
\end{proof}

\subsubsection{Completion of the proof}

Choose $r$ with $\beta<r<\rho$. By \eqref{eq:remainder-properties}, after
discarding finitely many indices we have
\begin{equation}
  \norm{w_j}_{A(G)}\le r<\rho.
  \label{eq:w-less-than-rho}
\end{equation}
Consider
\begin{equation}
  S_j:=\sum_{k=0}^{\infty}(-1)^kh_j^{-k-1}w_j^k.
  \label{eq:final-inverse-series}
\end{equation}
Using Lemma~\ref{lem:inverse-power-identity},
\begin{align*}
  \sum_{k=0}^{\infty}
  \norm{h_j^{-k-1}w_j^k}_{A(G)}
  &\le
  \sum_{k=0}^{\infty}
  \norm{h_j^{-k-1}}_{A(G)}\norm{w_j}_{A(G)}^k\\
  &\le
  \sum_{k=0}^{\infty}
  \rho^k\norm{h_j^{-k-1}}_{A(G)}\\
  &=\norm{(h_j-\rho z)^{-1}}_{A(G\times\T)}.
\end{align*}
The right-hand side is uniformly bounded by \eqref{eq:Vj-uniform-bound}. More
precisely, if the common bound there is $C$, then for every $N$
\[
  \sum_{k\ge N}\norm{h_j^{-k-1}w_j^k}_{A(G)}
  \le\left(\frac r\rho\right)^N
  \sum_{k\ge N}\rho^k\norm{h_j^{-k-1}}_{A(G)}
  \le C\left(\frac r\rho\right)^N.
\]
Hence the series \eqref{eq:final-inverse-series} converges absolutely, with tails
uniform in $j$, and the norms $\norm{S_j}_{A(G)}$ are uniformly bounded.

For its $N$th partial sum $S_{j,N}$, direct telescoping gives
\[
  (h_j+w_j)S_{j,N}
  =1+(-1)^Nh_j^{-N-1}w_j^{N+1}.
\]
Absolute convergence implies
$\norm{h_j^{-N-1}w_j^N}_{A(G)}\to0$. Since $\norm{w_j}_{A(G)}<\rho$ for all
large $j$,
\[
  \norm{h_j^{-N-1}w_j^{N+1}}_{A(G)}
  \le
  \norm{w_j}_{A(G)}
  \norm{h_j^{-N-1}w_j^N}_{A(G)}
  \longrightarrow0,
\]
and therefore the last term tends to zero as $N\to\infty$. Thus
\[
  (h_j+w_j)S_j=1.
\]
Since $f_j=h_j+w_j$, we obtain $S_j=f_j^{-1}$, with
$\sup_j\norm{f_j^{-1}}_{A(G)}<\infty$. This contradicts
\eqref{eq:bad-fixed-group} and proves
Proposition~\ref{prop:fixed-group-no-bad-sequence}.

Finally, if the constant in Theorem~\ref{thm:uniform-fourier-inversion} were not
uniform in the underlying compact group, the product construction at the beginning of
the proof would produce a bad sequence in one fixed compact group. This contradiction
proves Theorem~\ref{thm:uniform-fourier-inversion}.

\begin{remark}
Apart from the preliminary use of the qualitative Wiener lemma to identify individual
inverses, its decisive use in the compactness argument is the inversion of the limiting
resolvent $\mathcal R_{\rho}|_{K\times\T}$. No quantitative inversion estimate is
assumed at any stage. The strict inequality $\delta>1/2$
enters precisely through the mass gap \eqref{eq:key-mass-gap}, which allows one to
choose $\beta<\rho<\delta$. The use of the one-sided algebra $A_+$ is what makes the
proof valid for disconnected compact groups as well. If $G$ is connected, this
one-sided-spectrum step can be omitted: once $h_j-\rho z$ is invertible on
$G\times\T$, connectedness and
$\liminf_j\norm{h_j}_2\ge\delta>\rho$ force $\abs{h_j}>\rho$ everywhere for all
sufficiently large $j$.
\end{remark}

\subsection{Measure algebras}
\label{subsec:measure-algebra-positive}

We now transfer the result of the preceding subsection from Fourier algebras of
compact Abelian groups to measure algebras of arbitrary locally compact Abelian
groups. A direct application of Theorem~\ref{thm:uniform-fourier-inversion} is not
quite sufficient: after decomposing a measure into its discrete and continuous parts,
one needs simultaneous control of all powers of the inverse of the discrete part. The
following strengthened form of the argument from Subsection~\ref{subsec:fourier-algebra-positive}
is tailored precisely to this purpose.

\begin{proposition}[Budgeted resolvent estimate]
\label{prop:budgeted-resolvent}
For every $\delta>1/2$ there exists a constant $R_A(\delta)<\infty$ such that for
every compact Abelian group $G$, every $a\in A(G)$ and every $s\ge0$ satisfying
\begin{equation}
  \norm{a}_{A(G)}+s\le1,
  \qquad
  \inf_{x\in G}\abs{a(x)}\ge\delta,
  \label{eq:budgeted-assumptions}
\end{equation}
one has
\begin{equation}
  \norm{(a-sz)^{-1}}_{A(G\times\T)}\le R_A(\delta),
  \label{eq:budgeted-resolvent-bound}
\end{equation}
where $z$ denotes the standard character of the second factor $\T$. Moreover,
\begin{equation}
  \norm{(a-sz)^{-1}}_{A(G\times\T)}
  =\sum_{n=0}^{\infty}s^n\norm{a^{-n-1}}_{A(G)}.
  \label{eq:budgeted-power-identity}
\end{equation}
When $s=0$, the series in \eqref{eq:budgeted-power-identity} is understood as
$\norm{a^{-1}}_{A(G)}$. The constant $R_A(\delta)$ is independent of $G$, $a$ and $s$.
\end{proposition}

\begin{proof}
Suppose that the assertion fails for some fixed $\delta>1/2$. Then there are compact
Abelian groups $G_j$, functions $a_j\in A(G_j)$ and numbers $s_j\ge0$ such that
\[
  \norm{a_j}_{A(G_j)}+s_j\le1,
  \qquad
  \inf_{G_j}|a_j|\ge\delta,
\]
and
\[
  \norm{(a_j-s_jz)^{-1}}_{A(G_j\times\T)}\longrightarrow\infty.
\]
The inverses exist because
\[
  s_j\le1-\norm{a_j}_{A(G_j)}
  \le1-\delta<\delta\le|a_j|.
\]

We first reduce to one compact group, keeping the auxiliary circle variable. Put
\[
  G:=\prod_{j\ge1}G_j
\]
and let $\pi_j:G\to G_j$ be the coordinate projection. Define
\[
  A_j:=a_j\circ\pi_j.
\]
The pullback $\pi_j^*:A(G_j)\to A(G)$ is isometric, and so is the pullback by
$\pi_j\times\mathrm{id}_{\T}$. Hence
\[
  A_j(x)-s_jz
  =
  (a_j-s_jz)\circ(\pi_j\times\mathrm{id}_{\T}),
\]
and
\[
  \norm{(A_j-s_jz)^{-1}}_{A(G\times\T)}
  =
  \norm{(a_j-s_jz)^{-1}}_{A(G_j\times\T)}.
\]
Thus we may assume that all functions live on one compact group $G$.

Let $\Gamma_0$ be the subgroup of $\widehat G$ generated by the union of the
Fourier supports of all $a_j$. It is countable. Passing to the quotient
$G_0=G/\Gamma_0^\perp$ identifies each $a_j$ isometrically with an element of
$A(G_0)$; the same is true of $a_j-s_jz$ and its inverse on $G_0\times\T$.
Consequently we may also assume that $\widehat G$ is countable.

Passing to a subsequence, let
\[
  s_j\to s,
  \qquad
  \norm{a_j}_{A(G)}\to L,
  \qquad
  \norm{a_j}_2^2\to E.
\]
Apply Theorem~\ref{thm:profile-decomposition} to $(a_j)$:
\[
  a_j=h_j+w_j,
  \qquad
  \norm{w_j}_{A(G)}\to\beta,
\]
and write
\[
  P:=\sum_r\norm{G_r}_{A(G)}.
\]
The mass identity \eqref{eq:profile-mass-balance} gives $L=P+\beta$. Since
$L+s\le1$,
\begin{equation}
  P+\beta+s\le1.
  \label{eq:budgeted-total-mass}
\end{equation}
Moreover, $|a_j|\ge\delta$ implies $E\ge\delta^2$, whereas
\eqref{eq:profile-energy-balance} yields
\[
  E
  =
  \sum_r\norm{G_r}_2^2
  \le
  \left(\sum_r\norm{G_r}_2\right)^2
  \le
  P^2.
\]
Hence $P\ge\delta$, and \eqref{eq:budgeted-total-mass} gives
\begin{equation}
  \beta+s\le1-\delta<\delta.
  \label{eq:budgeted-mass-gap}
\end{equation}
Choose
\[
  \beta+s<\rho<\delta.
\]

Construct the limiting core $H$, the subgroups $K_j$ and $K$, and the embeddings
$\phi_j:G\to K_j$ exactly as in Subsection~\ref{subsec:fourier-algebra-positive}.
The proof of Lemma~\ref{lem:H-lower-bound} uses only
$\norm{w_j}_2\to0$ and $|a_j|\ge\delta$, and therefore gives
\[
  |H|\ge\delta
  \qquad\text{on }K.
\]
Thus
\[
  (H-\rho z)^{-1}\in A_+(K\times\T).
\]
Applying Lemma~\ref{lem:analytic-pullback} to $F=H-\rho z$ gives, for all
sufficiently large $j$,
\[
  V_j:=(h_j-\rho z)^{-1}\in A_+(G\times\T)
\]
and
\begin{equation}
  \sup_j\norm{V_j}_{A(G\times\T)}<\infty.
  \label{eq:budgeted-core-resolvent}
\end{equation}
By Lemma~\ref{lem:inverse-power-identity},
\begin{equation}
  \norm{V_j}_{A(G\times\T)}
  =
  \sum_{k=0}^{\infty}
  \rho^k\norm{h_j^{-k-1}}_{A(G)}.
  \label{eq:budgeted-core-power-series}
\end{equation}

Set
\[
  q_j(x,z):=w_j(x)-s_jz.
\]
The Fourier support of $w_j$ lies in degree $0$ of the $z$-variable, whereas
$s_jz$ lies in degree $1$. Hence these supports are disjoint and
\[
  \norm{q_j}_{A(G\times\T)}
  =
  \norm{w_j}_{A(G)}+s_j
  \longrightarrow\beta+s<\rho.
\]
For large $j$, therefore, $\norm{q_j}<\rho$. Consider
\[
  S_j
  :=
  \sum_{k=0}^{\infty}
  (-1)^kh_j^{-k-1}q_j^k.
\]
Using \eqref{eq:budgeted-core-power-series},
\[
  \sum_{k=0}^{\infty}
  \norm{h_j^{-k-1}q_j^k}
  \le
  \sum_{k=0}^{\infty}
  \norm{h_j^{-k-1}}\norm{q_j}^k
  \le
  \norm{V_j},
\]
so the series converges absolutely and its norms are uniformly bounded. For the
$N$th partial sum,
\[
  (h_j+q_j)S_{j,N}
  =
  1+(-1)^Nh_j^{-N-1}q_j^{N+1}.
\]
Since the series is absolutely convergent,
$\norm{h_j^{-N-1}q_j^N}\to0$, and thus
\[
  \norm{h_j^{-N-1}q_j^{N+1}}
  \le
  \norm{q_j}\,
  \norm{h_j^{-N-1}q_j^N}
  \longrightarrow0.
\]
Because $h_j+q_j=a_j-s_jz$, we obtain
\[
  S_j=(a_j-s_jz)^{-1}.
\]
Together with \eqref{eq:budgeted-core-resolvent}, this contradicts the choice of the
bad sequence and proves \eqref{eq:budgeted-resolvent-bound}.

Finally, for a single pair $(a,s)$ satisfying \eqref{eq:budgeted-assumptions},
\[
  s\le1-\delta<\delta\le|a(x)|
  \qquad(x\in G),
\]
so for every fixed $x$
\[
  \frac1{a(x)-sz}
  =
  \sum_{n=0}^{\infty}s^nz^na(x)^{-n-1}
  \qquad(|z|\le1).
\]
The left-hand side belongs to $A(G\times\T)$ by the first part of the proposition.
Uniqueness of Fourier coefficients in the $z$-variable therefore yields
\[
  \widehat{(a-sz)^{-1}}(\gamma,n)
  =
  s^n\widehat{a^{-n-1}}(\gamma),
  \qquad n\ge0,
\]
with no negative $z$-frequencies. Summing the absolute values proves
\eqref{eq:budgeted-power-identity}.
\end{proof}

We shall also use the following consequence of the theorem of Glicksberg and Wik
\cite{GlicksbergWik1972}, in precisely the form recorded in
\cite[Theorem~2.2 and Corollary~2.3]{OhryskoWasilewski2020}.

\begin{proposition}[Glicksberg--Wik]
\label{prop:glicksberg-wik}
Let $H$ be a locally compact Abelian group and let
$\mu=\mu_d+\mu_c\in M(H)$ be the decomposition of $\mu$ into its discrete and
continuous parts. Then
\begin{equation}
  \inf_{\gamma\in\widehat H}
  \abs{\widehat{\mu_d}(\gamma)}
  \ge
  \inf_{\gamma\in\widehat H}
  \abs{\widehat\mu(\gamma)}.
  \label{eq:glicksberg-wik-lower-bound}
\end{equation}
\end{proposition}

\begin{theorem}[Uniform norm-controlled inversion in measure algebras]
\label{thm:uniform-measure-inversion}
For every $\delta>1/2$ there exists a constant $C_M(\delta)<\infty$ with the
following property. For every locally compact Abelian group $H$ and every
$\mu\in M(H)$ satisfying
\begin{equation}
  \norm{\mu}_{M(H)}\le1,
  \qquad
  \inf_{\gamma\in\widehat H}\abs{\widehat\mu(\gamma)}\ge\delta,
  \label{eq:measure-inversion-assumptions}
\end{equation}
the measure $\mu$ is invertible in $M(H)$ and
\begin{equation}
  \norm{\mu^{-1}}_{M(H)}\le C_M(\delta).
  \label{eq:measure-inversion-conclusion}
\end{equation}
The constant $C_M(\delta)$ is independent of $H$.
\end{theorem}

\begin{proof}
Write
\begin{equation}
  \mu=\mu_d+\mu_c,
  \qquad
  \norm{\mu}_{M(H)}
  =\norm{\mu_d}_{M(H)}+\norm{\mu_c}_{M(H)},
  \label{eq:measure-decomposition}
\end{equation}
where the norm identity follows from the mutual singularity of the atomic and
non-atomic parts. By Proposition~\ref{prop:glicksberg-wik},
\begin{equation}
  \inf_{\gamma\in\widehat H}
  \abs{\widehat{\mu_d}(\gamma)}\ge\delta.
  \label{eq:discrete-part-lower-bound}
\end{equation}

Let $H_d$ denote the underlying group of $H$ equipped with the discrete topology.
The compact group
\[
  K:=\widehat{H_d}
\]
is canonically the Bohr compactification $b\widehat H$ of the dual group
$\widehat H$: the canonical homomorphism
\[
  \iota:\widehat H\longrightarrow K
\]
has dense range. We use this standard realization of the Bohr compactification; compare
\cite[Subsection~3.1.4]{Nikolski1999}. Fourier transformation gives canonical isometric
algebra identifications
\begin{equation}
  M_d(H)\cong\ell^1(H_d)\cong A(K).
  \label{eq:discrete-measures-fourier-algebra}
\end{equation}
Let $a\in A(K)$ correspond to $\mu_d$. Then
\begin{equation}
  \norm{a}_{A(K)}=\norm{\mu_d}_{M(H)},
  \label{eq:a-mud-norm}
\end{equation}
and, by construction of the embedding into the Bohr compactification,
\[
  a(\iota(\gamma))=\widehat{\mu_d}(\gamma)
  \qquad(\gamma\in\widehat H).
\]
Since $\iota(\widehat H)$ is dense in $K$, the lower bound
\eqref{eq:discrete-part-lower-bound} extends by continuity to all of $K$:
\begin{equation}
  \abs{a(x)}\ge\delta
  \qquad(x\in K).
  \label{eq:a-lower-bound-K}
\end{equation}

Set
\[
  s:=\norm{\mu_c}_{M(H)}.
\]
By \eqref{eq:measure-decomposition} and \eqref{eq:a-mud-norm},
\[
  \norm{a}_{A(K)}+s=\norm{\mu}_{M(H)}\le1.
\]
Proposition~\ref{prop:budgeted-resolvent} therefore gives
\begin{equation}
  \sum_{n=0}^{\infty}
  s^n\norm{a^{-n-1}}_{A(K)}
  \le R_A(\delta).
  \label{eq:weighted-inverse-powers-a}
\end{equation}
In particular, $a$ is invertible in $A(K)$; through
\eqref{eq:discrete-measures-fourier-algebra}, $\mu_d$ is invertible in $M_d(H)$.
Write
\[
  \nu_d:=\mu_d^{-1}.
\]
The isometry in \eqref{eq:discrete-measures-fourier-algebra} turns
\eqref{eq:weighted-inverse-powers-a} into
\begin{equation}
  \sum_{n=0}^{\infty}
  s^n\norm{\nu_d^{*(n+1)}}_{M(H)}
  \le R_A(\delta).
  \label{eq:weighted-inverse-powers-measure}
\end{equation}

Consider now
\begin{equation}
  \nu
  :=\sum_{n=0}^{\infty}
  (-1)^n\nu_d^{*(n+1)}*\mu_c^{*n}.
  \label{eq:measure-inverse-series}
\end{equation}
The series converges absolutely in $M(H)$, since
\begin{align*}
  \sum_{n=0}^{\infty}
  \norm{\nu_d^{*(n+1)}*\mu_c^{*n}}_{M(H)}
  &\le
  \sum_{n=0}^{\infty}
  s^n\norm{\nu_d^{*(n+1)}}_{M(H)}\\
  &\le R_A(\delta).
\end{align*}
Consequently,
\begin{equation}
  \norm{\nu}_{M(H)}\le R_A(\delta).
  \label{eq:measure-series-bound}
\end{equation}

For the $N$th partial sum $\nu_N$ of \eqref{eq:measure-inverse-series}, commutativity
and $\mu_d*\nu_d=\delta_0$ give the exact telescoping identity
\begin{equation}
  \bigl(\mu_d+\mu_c\bigr)*\nu_N
  =\delta_0
  +(-1)^N\nu_d^{*(N+1)}*\mu_c^{*(N+1)}.
  \label{eq:measure-telescoping}
\end{equation}
The norm of the last term is bounded by
\[
  s^{N+1}\norm{\nu_d^{*(N+1)}}_{M(H)}
  =s\Bigl(s^N\norm{\nu_d^{*(N+1)}}_{M(H)}\Bigr)
  \longrightarrow0,
\]
because the expression in parentheses is a term of the convergent series
\eqref{eq:weighted-inverse-powers-measure}; the case $s=0$ is immediate. Letting
$N\to\infty$ in \eqref{eq:measure-telescoping} yields
\[
  \mu*\nu=\delta_0.
\]
Thus $\nu=\mu^{-1}$, and \eqref{eq:measure-series-bound} proves
\eqref{eq:measure-inversion-conclusion} with
\[
  C_M(\delta):=R_A(\delta).
\]
\end{proof}

\begin{remark}
The budgeted form of Proposition~\ref{prop:budgeted-resolvent} is essential here. A
plain estimate for $\norm{\mu_d^{-1}}$ would lead to the usual Neumann condition
$\norm{\mu_d^{-1}}\norm{\mu_c}<1$, and such an inequality does not follow from
\eqref{eq:measure-inversion-assumptions}. What is needed instead is precisely the
weighted control of all inverse powers in
\eqref{eq:weighted-inverse-powers-measure}; the weight
$s=\norm{\mu_c}$ is exactly the one generated by the powers of the continuous part
in \eqref{eq:measure-inverse-series}.
\end{remark}

\subsection{Toeplitz operators}
\label{subsec:toeplitz-positive}

The final subsection of Section~\ref{sec:positive} treats the corresponding
norm-controlled inversion problem for Toeplitz operators on the Wiener algebra.
We now adopt the notation standard in the theory of Toeplitz operators. Put
\[
  W:=A(\T)
  =\left\{f(z)=\sum_{n\in\Z}\fhat(n)z^n:
  \sum_{n\in\Z}\abs{\fhat(n)}<\infty\right\}
\]
and
\[
  W^+:=\left\{g\in W:\widehat g(n)=0\text{ for }n<0\right\}.
\]
Let $P_+:W\to W^+$ be the Riesz projection and put $P_-:=I-P_+$. For
$f\in W$, let
\[
  T_f:W^+\longrightarrow W^+,
  \qquad
  T_fg=P_+(fg),
\]
be the Toeplitz operator with symbol $f$. Its matrix in the standard basis
$(1,z,z^2,\ldots)$ is $(\fhat(i-j))_{i,j\ge0}$. Before recalling the inversion criterion, let us fix the convention for the winding
number. If $f\in C(\T)$ does not vanish on $\T$, choose a continuous lift
$\vartheta:[0,2\pi]\to\R$ such that
\[
  \frac{f(e^{it})}{\abs{f(e^{it})}}=e^{i\vartheta(t)}
  \qquad (0\le t\le2\pi).
\]
Then
\begin{equation}
  \Ind(f)
  :=\frac{\vartheta(2\pi)-\vartheta(0)}{2\pi}\in\Z.
  \label{eq:winding-number-definition}
\end{equation}
Equivalently, $\Ind(f)$ is the topological degree of the map
$f/\abs f:\T\to\T$. With this convention, Krein's theorem, as recalled in
\cite[p.~396]{Nikolskaia2001}, gives
\begin{equation}
  T_f\text{ is invertible}
  \quad\Longleftrightarrow\quad
  \inf_{z\in\T}\abs{f(z)}>0
  \ \text{and}\ 
  \Ind(f)=0.
  \label{eq:krein-criterion}
\end{equation}

\begin{lemma}[Norm of a Toeplitz operator]
\label{lem:toeplitz-symbol-norm}
For every $f\in W$,
\[
  \norm{T_f}_{\mathcal L(W^+)}=\norm f_W.
\]
\end{lemma}

\begin{proof}
Since $P_+$ is contractive and $W$ is a Banach algebra,
$\norm{T_f}\le\norm f_W$. Conversely, for $N\ge0$,
\[
  \norm{T_fz^N}_{W^+}
  =\sum_{k\ge-N}\abs{\fhat(k)}
  \longrightarrow\sum_{k\in\Z}\abs{\fhat(k)}=\norm f_W,
\]
while $\norm{z^N}_{W^+}=1$.
\end{proof}

The estimate obtained in Subsection~\ref{subsec:largest-coefficient} immediately improves the explicit range in
\cite[Theorem~3.1]{Nikolskaia2001}. Recall that Nikolskaia obtained the bound
$(2\delta^2-1)^{-1}$ for general symbols when $\delta>1/\sqrt2$.

\begin{proposition}[Explicit estimate from the largest Fourier coefficient]
\label{prop:toeplitz-explicit-range}
Let
\[
  \delta>2-\sqrt2
\]
and let $f\in W$ satisfy
\[
  \norm f_W\le1,
  \qquad
  \inf_{z\in\T}\abs{f(z)}\ge\delta,
  \qquad
  \Ind(f)=0.
\]
Then
\begin{equation}
  \norm{T_f^{-1}}_{\mathcal L(W^+)}
  \le
  \frac{1}{\delta+\sqrt{2\delta-1}-1}.
  \label{eq:toeplitz-explicit-bound}
\end{equation}
\end{proposition}

\begin{proof}
The assumptions imply $\delta\le\norm f_\infty\le\norm f_W\le1$. Hence
Theorem~\ref{thm:improved-largest-coefficient} applies, and there is $n_0\in\Z$ such that
\[
  a:=\abs{\fhat(n_0)}
  =\norm{\fhat}_{\infty}
  \ge\Phi(\delta)>\frac12.
\]
After multiplication by a unimodular constant, write
\[
  f(z)=az^{n_0}+\psi(z),
  \qquad
  \norm\psi_W\le1-a<a.
\]
The homotopy $az^{n_0}+s\psi$, $0\le s\le1$, does not meet the origin, so
homotopy invariance of the winding number gives
\[
  \Ind(f)=n_0.
\]
Hence $n_0=0$. Therefore
\[
  T_f=aI+T_\psi,
  \qquad
  \norm{T_\psi}=\norm\psi_W\le1-a<a,
\]
and the Neumann series yields
\[
  \norm{T_f^{-1}}
  \le\frac{1}{a-\norm\psi_W}
  \le\frac{1}{2a-1}
  \le\frac{1}{2\Phi(\delta)-1},
\]
which is exactly \eqref{eq:toeplitz-explicit-bound}.
\end{proof}

\subsubsection{Why the Fourier-algebra argument does not apply directly}

Theorem~\ref{thm:uniform-fourier-inversion} cannot simply be applied to obtain a
Toeplitz estimate for every $\delta>1/2$. It gives a uniform bound for
$\norm{f^{-1}}_W$, whereas the object to be controlled here is
$\norm{T_f^{-1}}$. In general,
\[
  T_f^{-1}\ne T_{f^{-1}},
\]
and the inverse of a Toeplitz operator need not itself be a Toeplitz operator. The
underlying obstruction is the failure of multiplicativity of the Toeplitz map:
$T_aT_b-T_{ab}$ is generally non-zero when both $a$ and $b$ have positive and
negative Fourier frequencies. Thus norm-controlled inversion in the Banach algebra
$W$ controls the inverse of the symbol, but not directly the inverse of the associated
Toeplitz operator.

What does survive from Subsection~\ref{subsec:fourier-algebra-positive} is the compactness mechanism behind its proof.
The same profile decomposition, limiting subgroup, mass gap and one-sided lifting
argument can be used to construct uniformly controlled logarithms of the symbols. A
controlled logarithm is exactly what is needed for a controlled Wiener--Hopf
factorization.

Put
\[
  W^-:=\left\{g\in W:\widehat g(n)=0\text{ for }n>0\right\}.
\]
Thus $W^-$ contains the constant coefficient, whereas the range of $P_-=I-P_+$
consists of the elements of $W^-$ with zero constant term.
We shall use the following elementary product rules.

\begin{lemma}[Toeplitz product rules]
\label{lem:toeplitz-product-rules}
For $a,b\in W$ one has
\[
  T_aT_b=T_{ab}
\]
whenever either $b\in W^+$ or $a\in W^-$. Consequently, if $f=e^\ell$ with
$\ell\in W$, then $T_f$ is invertible and
\begin{equation}
  \norm{T_f^{-1}}
  \le e^{\norm\ell_W}.
  \label{eq:toeplitz-controlled-log}
\end{equation}
\end{lemma}

\begin{proof}
If $b\in W^+$, then $bg\in W^+$ for every $g\in W^+$, and hence
\[
  T_aT_bg=P_+(abg)=T_{ab}g.
\]
If $a\in W^-$, then $P_+\bigl(aP_-(bg)\bigr)=0$, so the same identity follows.

Now write
\[
  \ell=\ell_-+\ell_+,
\]
where $\ell_-$ has only strictly negative Fourier frequencies and $\ell_+$ has only
non-negative ones. Then $e^{\ell_-}\in W^-$ and $e^{\ell_+}\in W^+$, whence
\[
  T_f=T_{e^{\ell_-}}T_{e^{\ell_+}}
\]
and
\[
  T_f^{-1}=T_{e^{-\ell_+}}T_{e^{-\ell_-}}.
\]
By Lemma~\ref{lem:toeplitz-symbol-norm},
\[
  \norm{T_f^{-1}}
  \le\norm{e^{-\ell_+}}_W\norm{e^{-\ell_-}}_W
  \le e^{\norm{\ell_+}_W+\norm{\ell_-}_W}
  =e^{\norm\ell_W}.
\]
\end{proof}

\begin{factorizationlemma}[Bohr--van Kampen--Arens--Royden factorization]
Let $Y$ be a compact connected Abelian group. If $F\in A(Y)$ is invertible, then
there are a character $\chi\in\widehat Y$ and a function $G\in A(Y)$ such that
\begin{equation}
  F=\chi e^G.
  \label{eq:toeplitz-bvk-arens}
\end{equation}
\end{factorizationlemma}

\begin{proof}
The classical Bohr--van Kampen theorem says that every non-vanishing continuous
function on a compact connected Abelian group is, up to multiplication by a
character, an exponential of a continuous function; see the formulation and
function-algebra extension in \cite{Gorin1970}. Hence
\[
  F=\chi e^g
  \qquad
  (\chi\in\widehat Y,\ g\in C(Y)).
\]
Thus the class of $F\overline\chi$ is trivial in
$C(Y)^{-1}/\exp C(Y)$. The Arens--Royden theorem identifies, through the Gelfand
transform, the quotient of the invertible group of a commutative Banach algebra by
its exponential component with the corresponding quotient for the algebra of
continuous functions on its maximal ideal space; see \cite[pp.~21--23]{Arens1963}.
Since the maximal ideal space of $A(Y)$ is $Y$, the class of
$F\overline\chi$ is already trivial in $A(Y)^{-1}/\exp A(Y)$. Hence
$F\overline\chi=e^G$ for some $G\in A(Y)$.
\end{proof}

\subsubsection{The optimal positive range}

\begin{theorem}[Norm-controlled inversion for Toeplitz operators]
\label{thm:toeplitz-half-threshold}
For every $\delta>1/2$ there exists a constant $C_T(\delta)<\infty$ such that, for
all $f\in W$ satisfying
\begin{equation}
  \norm f_W\le1,
  \qquad
  \inf_{z\in\T}\abs{f(z)}\ge\delta,
  \qquad
  \Ind(f)=0,
  \label{eq:toeplitz-main-assumptions}
\end{equation}
the operator $T_f:W^+\to W^+$ is invertible and
\begin{equation}
  \norm{T_f^{-1}}_{\mathcal L(W^+)}\le C_T(\delta).
  \label{eq:toeplitz-main-conclusion}
\end{equation}
\end{theorem}

\begin{proof}
Fix $\delta>1/2$ and suppose that the assertion fails. Then there is a sequence
$(f_j)\subset W$ such that
\begin{equation}
  \norm{f_j}_W\le1,
  \qquad
  \abs{f_j}\ge\delta,
  \qquad
  \Ind(f_j)=0,
  \qquad
  \norm{T_{f_j}^{-1}}\longrightarrow\infty.
  \label{eq:toeplitz-bad-sequence}
\end{equation}
Since $\norm{f_j}_2\le\norm{f_j}_W\le1$, passing to a further subsequence, we may
also assume that $\norm{f_j}_2^2$ converges.

We now use the machinery of Subsection~\ref{subsec:fourier-algebra-positive} with $G=\T$. By
Theorem~\ref{thm:profile-decomposition}, after passing to a subsequence we have
\[
  f_j=h_j+w_j,
  \qquad
  \norm{w_j}_W\longrightarrow\beta,
  \qquad
  \norm{w_j}_2\longrightarrow0.
\]
Using exactly the limiting-core construction
\eqref{eq:limiting-core-H}--\eqref{eq:limiting-subgroup-K}, we obtain a compact
Abelian group
\[
  X=\T\times\T^I,
\]
compact subgroups $K_j,K\subset X$, homomorphisms
$\phi_j:\T\to K_j$, and a function $H\in A(X)$ such that
\[
  H\circ\phi_j=h_j.
\]
Lemma~\ref{lem:H-lower-bound} gives
\begin{equation}
  \abs{H}\ge\delta
  \qquad\text{on }K,
  \label{eq:toeplitz-H-lower}
\end{equation}
and the mass-gap argument \eqref{eq:key-mass-gap} gives
\begin{equation}
  \beta\le1-\delta<\delta.
  \label{eq:toeplitz-mass-gap}
\end{equation}
Choose once and for all
\begin{equation}
  \beta<\rho<\delta.
  \label{eq:toeplitz-rho}
\end{equation}

For large $j$ we have $\norm{w_j}_W<\rho$. Hence, for $0\le s\le1$,
\[
  \abs{f_j-sw_j}
  \ge\delta-s\norm{w_j}_\infty
  >\delta-\rho>0.
\]
The functions $f_j$ and $h_j=f_j-w_j$ are therefore homotopic through non-vanishing
functions, and
\begin{equation}
  \Ind(h_j)=\Ind(f_j)=0,
  \qquad
  \inf_{\T}\abs{h_j}>\delta-\rho.
  \label{eq:toeplitz-core-index}
\end{equation}

At this point the Toeplitz problem requires one additional ingredient that was not
needed for ordinary inversion in Subsection~\ref{subsec:fourier-algebra-positive}: the limiting core must have a
logarithm. We verify this from the zero-index assumption.

Let $\mathcal L=K^{\perp}\subset\widehat X$ be the group of stable relations from
\eqref{eq:limiting-subgroup-K}. Since here
\[
  \widehat X=\Z\oplus\Z^{(I)},
\]
the subgroup $\mathcal L$ is pure.\footnote{A subgroup $L$ of an Abelian group
$A$ is called \emph{pure} if $L\cap mA=mL$ for every $m\ge1$. Since
$\widehat X$ is torsion-free, this is equivalent here to saying that
$m\lambda\in L$ for some non-zero integer $m$ implies $\lambda\in L$.} Indeed, for
$\lambda=(n,(a_r))\in\widehat X$ the restriction of $\lambda$ to $K_j$ corresponds,
under $\phi_j$, to the integer
\[
  d_j(\lambda)=n+\sum_{r\in I}a_r\gamma_j^{(r)}.
\]
If $m\lambda\in\mathcal L$ for some non-zero integer $m$, then
$m d_j(\lambda)=0$ for all sufficiently large $j$, hence $d_j(\lambda)=0$ there and
$\lambda\in\mathcal L$. Consequently, by \cite[Theorem~2.1.2]{Rudin1990},
\[
  \widehat K\cong\widehat X/\mathcal L.
\]
This quotient is torsion-free, and therefore $K$ is connected by
\cite[Theorem~2.5.6(c)]{Rudin1990}.

Applying the Bohr--van Kampen--Arens--Royden factorization above to $H|_K$, we obtain
\[
  H|_K=\chi e^G.
\]
We claim that $\chi=1$. Choose a representative
$\widetilde\chi\in\widehat X$ of $\chi\in\widehat K\cong\widehat X/\mathcal L$,
and use the metric quotient property of the restriction $A(X)\to A(K)$ established
in Subsection~\ref{subsec:fourier-algebra-positive} to lift $G$ to some $\widetilde G\in A(X)$. Then
\[
  R:=H-\widetilde\chi e^{\widetilde G}
\]
vanishes on $K$. Lemma~\ref{lem:restriction-norm-stability} therefore yields
\[
  \norm{R|_{K_j}}_{A(K_j)}\longrightarrow0.
\]
After pulling back by $\phi_j$ and setting
\[
  q_j:=(\widetilde\chi e^{\widetilde G})\circ\phi_j,
\]
we obtain $\norm{q_j-h_j}_W\to0$. In particular,
\[
  \norm{q_j-h_j}_{\infty}
  \le
  \norm{q_j-h_j}_{W}
  \longrightarrow0.
\]
By \eqref{eq:toeplitz-core-index}, $\inf_{\T}|h_j|>\delta-\rho$ for all large
$j$. Hence the straight-line homotopy from $h_j$ to $q_j$ stays non-vanishing for
all sufficiently large $j$, and therefore $q_j$ and $h_j$ have the same winding
number. The exponential factor in $q_j$ has winding number zero, whereas
$\widetilde\chi\circ\phi_j(z)=z^{d_j(\widetilde\chi)}$. Hence
\[
  0=\Ind(q_j)=d_j(\widetilde\chi)
\]
for all sufficiently large $j$. Thus $\widetilde\chi\in\mathcal L$, so its class
$\chi$ in $\widehat K$ is trivial. We have proved that
\begin{equation}
  H|_K=e^G
  \qquad\text{for some }G\in A(K).
  \label{eq:toeplitz-H-exponential}
\end{equation}

We now return to the one-sided algebra $A_+$ introduced in Subsection~\ref{subsec:fourier-algebra-positive}. On
$K\times\T$ consider
\[
  \mathcal R_\rho(y,z):=H(y)-\rho z.
\]
By \eqref{eq:toeplitz-H-exponential},
\[
  \mathcal U_K
  :=G-\sum_{n=1}^{\infty}\frac1n
     \left(\frac{\rho z}{H}\right)^n
\]
is a well-defined element of $A_+(K\times\T)$ and
\begin{equation}
  e^{\mathcal U_K}=H-\rho z.
  \label{eq:toeplitz-limit-log}
\end{equation}
Indeed, the spectral radius of
\[
  T:=\frac{\rho z}{H}
\]
in $A(K\times\T)$ is at most $\rho/\delta<1$. Choose $r$ with
$r(T)<r<1$. By the spectral-radius formula there is $N_0$ such that
$\norm{T^n}^{1/n}\le r$ for every $n\ge N_0$, hence
$\norm{T^n}\le r^n$ for $n\ge N_0$. Therefore
\[
  \sum_{n=1}^{\infty}\frac{\norm{T^n}}{n}<\infty,
\]
which proves norm convergence of the logarithmic series. No estimate
$\norm{\rho/H}<1$ is required.

By the norm-preserving one-sided lifting property
\eqref{eq:positive-spectrum-lift}, choose
$\mathcal U\in A_+(X\times\T)$ whose restriction to $K\times\T$ is
$\mathcal U_K$. Put
\[
  D:=e^{\mathcal U}-(H-\rho z).
\]
Then $D|_{K\times\T}=0$. Lemma~\ref{lem:restriction-norm-stability}
gives
\begin{equation}
  \norm{D|_{K_j\times\T}}_{A(K_j\times\T)}\longrightarrow0.
  \label{eq:toeplitz-log-defect}
\end{equation}
Let
\[
  \mathcal U_j^{(0)}:=\mathcal U|_{K_j\times\T},
  \qquad
  F_j:=(H-\rho z)|_{K_j\times\T}.
\]
Then
\[
  E_j:=e^{-\mathcal U_j^{(0)}}F_j-1
  =
  -e^{-\mathcal U_j^{(0)}}
  D|_{K_j\times\T}.
\]
Restriction is contractive and
$\norm{\mathcal U_j^{(0)}}\le\norm{\mathcal U}$, so
\[
  \norm{E_j}
  \le
  e^{\norm{\mathcal U}}\,
  \norm{D|_{K_j\times\T}}
  \longrightarrow0
\]
by \eqref{eq:toeplitz-log-defect}. For large $j$ we have $\norm{E_j}<1/2$,
so
\[
  C_j:=\log(1+E_j)
  =\sum_{m=1}^{\infty}\frac{(-1)^{m+1}}{m}E_j^m
\]
converges in $A_+(K_j\times\T)$ and $\norm{C_j}\le\log2$. Thus
\[
  \mathcal U_j:=\mathcal U_j^{(0)}+C_j
\]
satisfies
\begin{equation}
  e^{\mathcal U_j}=F_j,
  \qquad
  \sup_j\norm{\mathcal U_j}_{A(K_j\times\T)}<\infty.
  \label{eq:toeplitz-exact-logs-kj}
\end{equation}
Pulling back by $\phi_j\times\operatorname{id}_{\T}$, we may regard
$\mathcal U_j$ as an element of $A_+(\T\times\T)$ satisfying
\begin{equation}
  e^{\mathcal U_j(t,z)}=h_j(t)-\rho z,
  \qquad
  \sup_j\norm{\mathcal U_j}_{A(\T\times\T)}<\infty.
  \label{eq:toeplitz-bivariate-log}
\end{equation}

Finally, if
\[
  V(t,z)=\sum_{n=0}^{\infty}V_n(t)z^n\in A_+(\T\times\T)
\]
and $u\in W$ with $\norm u_W\le1$, then
\[
  \Psi_u(V):=\sum_{n=0}^{\infty}V_nu^n
\]
defines a contractive homomorphism $A_+(\T\times\T)\to W$. Indeed,
\[
  \norm{\Psi_u(V)}_W
  \le\sum_{n\ge0}\norm{V_n}_W\norm{u}_W^n
  \le\sum_{n\ge0}\norm{V_n}_W
  =\norm{V}_{A(\T\times\T)}.
\]
Multiplicativity follows from the absolute convergence of the corresponding Cauchy
products. Since
$\norm{w_j}_W<\rho$, take
\[
  u_j:=-\frac{w_j}{\rho}
\]
and set
\[
  \ell_j:=\Psi_{u_j}(\mathcal U_j).
\]
By \eqref{eq:toeplitz-bivariate-log},
\[
  \sup_j\norm{\ell_j}_W<\infty,
\]
and, because $\Psi_{u_j}$ is a continuous homomorphism,
\[
  e^{\ell_j}
  =\Psi_{u_j}(e^{\mathcal U_j})
  =h_j-\rho u_j
  =h_j+w_j
  =f_j.
\]
Lemma~\ref{lem:toeplitz-product-rules} now gives
\[
  \sup_j\norm{T_{f_j}^{-1}}<\infty,
\]
contradicting \eqref{eq:toeplitz-bad-sequence}. This proves the theorem.
\end{proof}

\begin{remark}
Theorem~\ref{thm:toeplitz-half-threshold} closes the interval left open by the
largest-coefficient method. The argument does not turn the estimate for
$\norm{f^{-1}}_W$ from Subsection~\ref{subsec:fourier-algebra-positive} into a Toeplitz estimate. Rather, it reuses the
profile compactness and the mass gap from that subsection and adds the zero-index
information in order to obtain a uniformly controlled logarithm. The sharpness can also be read directly from Nikolski's formula for the analytic
Wiener algebra $W^+$, where the optimal inversion majorant on the closed disc is
$(2\delta-1)^{-1}$ for $1/2<\delta\le1$; see \cite[p.~1933]{Nikolski1999}.
For analytic symbols $T_f$ is multiplication by $f$, hence
$\norm{T_f^{-1}}=\norm{f^{-1}}_{W^+}$. Thus the threshold $1/2$ is optimal.
\end{remark}

\section{Sharpness at the critical threshold}
\label{sec:sharpness}

The question of uniform control of inverse norms in Wiener-type algebras predates the
identification of the critical threshold.  As noted by Nikolski in his historical account
\cite[Subsection~0.5]{Nikolski1999}, Stafney appears to have been the first to point out
that inverse norms in the Wiener algebra can be unbounded under simultaneous uniform
control of the algebra norm and separation of the range from zero
\cite{Stafney1967}.  More precisely, his result implies that $c_1(\delta,A(\mathbb T),
\mathbb T)=\infty$ for some $\delta>0$, but his argument does not determine any
specific value of $\delta$.  Shapiro subsequently obtained related counterexamples by a
more constructive method \cite{Shapiro1979}.  These early results revealed the failure of
uniform norm control, but did not identify the critical constant.

The sharpness of the threshold $1/2$ was already established by Nikolski
\cite[Theorem~3.2.2 and Corollary~3.2.5]{Nikolski1999}.  If $H$ is an infinite
locally compact Abelian group and $\mathcal A$ is one of $M(H)$, $M_d(H)$, or
$L^1(H)+\mathbb C\delta_0$, the visible evaluations are indexed by $\widehat H$,
and his result gives
\[
  c_1(\delta;\mathcal A,\widehat H)\ge \frac{1}{2\delta-1},
  \qquad \frac12<\delta\le1,
\]
and
\[
  c_1(\delta;\mathcal A,\widehat H)=\infty,
  \qquad 0<\delta\le\frac12.
\]
Thus the critical constant is at least $1/2$, with failure of norm control already
at the endpoint. Taking the discrete group $H=\widehat G$, so that
$\widehat H\cong G$, includes the Fourier algebra $A(G)$ of every infinite
compact Abelian group.

Nikolski's proof is indirect and uses a passage to the Bohr compactification from
\cite[Subsections~3.1--3.2]{Nikolski1999}.  On a nondiscrete group one starts
with a \v{S}re\u{\i}der measure $\nu$ of norm one, with real visible transform and full
Gelfand spectrum equal to the closed unit disc.  For $1/2<\delta<1$, with $\nu^2:=\nu*\nu$, the element
\[
  \mu_\delta=\delta\delta_0+(1-\delta)\nu^2
\]
has norm one and visible transform contained in $[\delta,1]$, whereas spectral
mapping gives
\[
  \sigma(\mu_\delta)=\delta+(1-\delta)\overline{\mathbb D}.
\]
Consequently $\norm{\mu_\delta^{-1}}\ge(2\delta-1)^{-1}$.  For a general infinite group the construction is carried out on the Bohr
compactification and transferred back using almost-periodic approximation, the
Bochner characterization, and semicontinuity of inversion.  This argument is abstract: it depends on \v{S}re\u{\i}der measures and the invisible
spectrum rather than on explicit elements of the algebra. Nevertheless, combined
with the qualitative inversion theorem of Ohrysko and Wasilewski
\cite[Theorem~2.4]{OhryskoWasilewski2020}, Nikolski's approach explains the origin
of the constant $1/2$: the qualitative invertibility threshold for measure algebras
$M(H)$ on nondiscrete groups becomes, through this transfer, a quantitative
threshold for norm-controlled inversion in Fourier algebras $A(G)$ of infinite
compact groups. The purpose of the present section is to give a direct and independent
construction. It is formulated entirely in $A(G)$ in the compact case and then
transferred to the unitization of $L^1$ by a concrete smearing procedure.

\subsection{Infinite compact groups}
\label{subsec:negative-compact-fourier}

We first give the direct compact-group construction. Let $G$ be an infinite compact
Abelian group. We construct the examples directly in $A(G)$, using the algebraic
structure of $\widehat G$ to choose the Fourier frequencies.

\begin{theorem}[Sharpness of the threshold $1/2$]
\label{thm:sharp-half-compact}
Let $G$ be an infinite compact Abelian group. There exists a sequence
$(f_n)\subset A(G)$ such that
\begin{equation}
  \norm{f_n}_{A(G)}\le1,
  \qquad
  \inf_{x\in G}\abs{f_n(x)}\ge\frac12,
  \qquad
  \norm{f_n^{-1}}_{A(G)}\longrightarrow\infty.
  \label{eq:sharp-half-main}
\end{equation}
In particular, norm-controlled inversion fails at $\delta=1/2$ on every infinite
compact Abelian group.
\end{theorem}

\begin{remark}
The assumption that $G$ is infinite is necessary. If $G$ is finite, then $A(G)$ is
finite-dimensional, and the set
\[
  \left\{f\in A(G):\norm{f}_{A(G)}\le1,
  \ \min_{x\in G}\abs{f(x)}\ge\frac12\right\}
\]
is compact. Since inversion is continuous on the group of invertible elements, the
norms of the inverses are uniformly bounded on this set.
\end{remark}

The construction depends on a simple dichotomy for the two-torsion subgroup
\[
  \widehat G[2]:=\{\gamma\in\widehat G:\gamma^2=1\}.
\]

\begin{definition}
More generally, if $\Gamma$ is an Abelian group written multiplicatively, a finite
family $\gamma_1,\ldots,\gamma_n\in\Gamma$ will be called
\emph{strongly dissociated} if the map
\[
  \{-1,0,1\}^n\ni(\varepsilon_1,\ldots,\varepsilon_n)
  \longmapsto
  \gamma_1^{\varepsilon_1}\cdots\gamma_n^{\varepsilon_n}
\]
is injective.
\end{definition}

\begin{lemma}[Structural dichotomy]
\label{lem:structural-dichotomy}
Let $\Gamma$ be an infinite Abelian group and put
\[
  \Gamma[2]:=\{\gamma\in\Gamma:\gamma^2=1\}.
\]
Exactly one of the following two alternatives occurs.
\begin{enumerate}[label=\textup{(\roman*)}]
  \item $\Gamma[2]$ is finite. Then $\Gamma$ contains an infinite sequence
  $\gamma_1,\gamma_2,\ldots$ every finite initial segment of which is strongly
  dissociated.
  \item $\Gamma[2]$ is infinite. Then $\Gamma[2]$, viewed as a vector space
  over $\mathbb F_2$, contains an infinite linearly independent sequence.
\end{enumerate}
\end{lemma}

\begin{proof}
If $\Gamma[2]$ is infinite, then it has infinite dimension over $\mathbb F_2$,
and (ii) follows immediately.

Assume that $\Gamma[2]$ is finite. Since $\Gamma$ is infinite, choose
$\gamma_1\notin\Gamma[2]$; then the one-element
family $(\gamma_1)$ is strongly dissociated. Suppose inductively that
$\gamma_1,\ldots,\gamma_n$ have been chosen strongly dissociated. Put
\[
  D_n:=\left\{
  \gamma_1^{\eta_1}\cdots\gamma_n^{\eta_n}:
  \eta_j\in\{-2,-1,0,1,2\}\right\}.
\]
The set $D_n$ is symmetric under inversion. Moreover,
\[
  D_n^{1/2}:=\{\gamma\in\Gamma:\gamma^2\in D_n\}
\]
is finite, because every non-empty fiber of the homomorphism
$\gamma\mapsto\gamma^2$ is a coset of the finite group $\Gamma[2]$. Hence we may
choose
\[
  \gamma_{n+1}\notin D_n\cup D_n^{1/2}.
\]
If two $\{-1,0,1\}$-representations involving $\gamma_{n+1}$ were equal, then after
division one would obtain
\[
  \gamma_{n+1}^{a}=\eta,
  \qquad
  a\in\{-2,-1,0,1,2\},
  \quad \eta\in D_n.
\]
The case $a=0$ is ruled out by the induction hypothesis, while $a=\pm1$ would imply
$\gamma_{n+1}\in D_n$ and $a=\pm2$ would imply
$\gamma_{n+1}\in D_n^{1/2}$. Thus the extended family is strongly dissociated.
\end{proof}

\subsubsection{Finite two-torsion: a Bessel construction}

Assume first that $\widehat G[2]$ is finite and choose a strongly dissociated sequence
$\gamma_1,\gamma_2,\ldots$ as above. For $x\in G$ set
\[
  c_j(x):=\operatorname{Re}\gamma_j(x)
  =\frac12\bigl(\gamma_j(x)+\overline{\gamma_j(x)}\bigr),
  \qquad
  d_n:=\frac1n\sum_{j=1}^n c_j.
\]
Strong dissociation implies that the frequencies $\gamma_j$ and $\gamma_j^{-1}$ are
pairwise distinct, and hence
\[
  \norm{d_n}_{A(G)}=1.
\]
Put
\[
  r_n:=\log n,
  \qquad
  \tau_n:=\frac{\log n}{n},
  \qquad
  b_n:=\exp(i r_nd_n).
\]
Since $d_n$ is real-valued,
\[
  \abs{b_n(x)}=1
  \qquad(x\in G),
\]
and the power-series expansion of the exponential in the Banach algebra $A(G)$ gives
\begin{equation}
  \norm{b_n}_{A(G)}
  \le \exp\bigl(r_n\norm{d_n}_{A(G)}\bigr)=n.
  \label{eq:bn-upper-case1}
\end{equation}

Define
\[
  \beta_n:=\frac{1}{2(n+1)},
  \qquad
  \alpha_n:=\frac12+\beta_n=\frac{n+2}{2(n+1)},
\]
and
\begin{equation}
  f_n:=\alpha_n-\beta_nb_n.
  \label{eq:fn-case1}
\end{equation}
Then \eqref{eq:bn-upper-case1} gives
\[
  \norm{f_n}_{A(G)}
  \le\alpha_n+\beta_n\norm{b_n}_{A(G)}
  \le\alpha_n+n\beta_n=1,
\]
while the unimodularity of $b_n$ gives
\[
  \abs{f_n(x)}
  =\abs{\alpha_n-\beta_nb_n(x)}
  \ge\alpha_n-\beta_n
  =\frac12.
\]
It remains to prove that the inverse norms diverge.

Let
\[
  S_n:=\left\{
  \gamma_1^{\varepsilon_1}\cdots\gamma_n^{\varepsilon_n}:
  \varepsilon_j\in\{-1,0,1\}\right\}.
\]
By strong dissociation, every $\gamma\in S_n$ has a unique representation of this
form. Write
\[
  s(\gamma):=\#\{j:\varepsilon_j\ne0\}.
\]
Define a functional $\Lambda_n\in A(G)^*$ by
\[
  \Lambda_n(f)
  :=\sum_{\gamma\in\widehat G}\omega_n(\gamma)\widehat f(\gamma),
  \qquad
  \omega_n(\gamma):=
  \begin{cases}
    i^{-s(\gamma)},&\gamma\in S_n,\\
    0,&\gamma\notin S_n.
  \end{cases}
\]
Since $\abs{\omega_n}\le1$,
\begin{equation}
  \norm{\Lambda_n}\le1.
  \label{eq:lambda-contractive-case1}
\end{equation}
The functional $\Lambda_n$ cancels the phases in the low-frequency product
$Q_{n,u}$ defined below. Frequencies arising from the remaining Bessel terms are
controlled by the error estimate rather than by an exact phase cancellation.

We use the Jacobi--Anger expansion
\begin{equation}
  e^{iu\cos\theta}
  =\sum_{m\in\Z}i^mJ_m(u)e^{im\theta},
  \label{eq:jacobi-anger}
\end{equation}
where $J_m$ denotes the Bessel function of the first kind; see
\cite[Chapter~10, Section~10.12, in particular Eqs.~(10.12.1)--(10.12.3)]{NIST2010}.
Thus, for every character
$\gamma\in\widehat G$,
\[
  e^{iu\operatorname{Re}\gamma}
  =\sum_{m\in\Z}i^mJ_m(u)\gamma^m
  \quad\text{in }A(G).
\]
Consequently, for every integer $k\ge0$ and $u=k\tau_n$,
\begin{equation}
  b_n^k
  =\prod_{j=1}^n e^{iu\operatorname{Re}\gamma_j}.
  \label{eq:bn-power-factorization}
\end{equation}

For $\gamma\in\{\gamma_1,\ldots,\gamma_n\}$ and $0\le u\le1/2$ put
\[
  q_{\gamma,u}
  :=J_0(u)+iJ_1(u)(\gamma+\gamma^{-1}),
\]
and let
\[
  r_{\gamma,u}:=e^{iu\operatorname{Re}\gamma}-q_{\gamma,u}.
\]
Set
\[
  A(u):=J_0(u)+2J_1(u),
  \qquad
  E(u):=2\sum_{m\ge2}\abs{J_m(u)}.
\]
Then
\[
  \norm{q_{\gamma,u}}_{A(G)}
  =\abs{J_0(u)}+2\abs{J_1(u)},
  \qquad
  \norm{r_{\gamma,u}}_{A(G)}\le E(u).
\]
The elementary Bessel estimates needed below are
\begin{equation}
  E(u)\le\frac12u^2,
  \qquad
  A(u)\ge1,
  \qquad
  \log A(u)\ge u-u^2
  \quad(0\le u\le1/2).
  \label{eq:bessel-estimates}
\end{equation}
For completeness, we derive the quantitative bounds needed below from the standard
power series for $J_m$; see
\cite[Chapter~10, Section~10.2(ii), Eq.~(10.2.2)]{NIST2010}. It gives
\[
  \abs{J_m(u)}
  \le
  e^{u^2/4}\frac{(u/2)^m}{m!}.
\]
Hence, for $0\le u\le1/2$,
\[
  E(u)
  \le
  2e^{u^2/4}
  \sum_{m\ge2}\frac{(u/2)^m}{m!}
  \le
  2e^{1/16}\frac{u^2}{8}e^{u/2}
  \le\frac{u^2}{2}.
\]
Moreover,
\[
  J_0(u)\ge1-\frac{u^2}{4},
  \qquad
  J_1(u)\ge\frac u2-\frac{u^3}{16}.
\]
In particular, $J_0(u),J_1(u)\ge0$ on $[0,1/2]$, and hence
$\norm{q_{\gamma,u}}_{A(G)}=A(u)$. Moreover,
\[
  A(u)
  \ge
  1+u-\frac{u^2}{4}-\frac{u^3}{8}
  \ge1.
\]
Using $\log(1+t)\ge t-t^2/2$ for $0\le t\le1$ with
\[
  t=u-\frac{u^2}{4}-\frac{u^3}{8},
\]
we obtain
\[
  \log A(u)
  \ge
  t-\frac{t^2}{2}.
\]
A direct expansion gives
\[
  t-\frac{t^2}{2}-(u-u^2)
  =
  \frac{u^2}{128}
  \left(
    32+16u+12u^2-4u^3-u^4
  \right)\ge0
\]
for $0\le u\le1/2$, proving the last inequality in
\eqref{eq:bessel-estimates}.

Let
\[
  Q_{n,u}:=\prod_{j=1}^n q_{\gamma_j,u}.
\]
Strong dissociation makes the $\{-1,0,1\}$-representation of every frequency in
$Q_{n,u}$ unique. Therefore the phase correction in $\Lambda_n$ is exact:
\begin{equation}
  \Lambda_n(Q_{n,u})=A(u)^n.
  \label{eq:lambda-Q-exact}
\end{equation}
On the other hand, telescoping the product in \eqref{eq:bn-power-factorization} and
using \eqref{eq:bessel-estimates} gives
\begin{equation}
  \norm{b_n^k-Q_{n,u}}_{A(G)}
  \le nE(u)\bigl(A(u)+E(u)\bigr)^{n-1}.
  \label{eq:bessel-product-error}
\end{equation}
Consequently, whenever
\[
  0\le u\le\frac12,
  \qquad
  nu^2\le\frac14,
\]
we obtain
\begin{equation}
  \operatorname{Re}\Lambda_n(b_n^k)
  \ge\frac12A(u)^n.
  \label{eq:lambda-bn-lower-intermediate}
\end{equation}
Indeed, by \eqref{eq:bessel-estimates}, $nE(u)\le1/8$, and the quotient of the
right-hand side of \eqref{eq:bessel-product-error} by $A(u)^n$ is at most
\[
  \frac18e^{1/8}<\frac12.
\]

Choose positive integers $K_n$ such that
\begin{equation}
  K_n\longrightarrow\infty,
  \qquad
  \frac{K_n^2(\log n)^2}{n}\longrightarrow0.
  \label{eq:Kn-negative}
\end{equation}
For example, for all sufficiently large $n$ one may take
$K_n=\lfloor n^{1/4}\rfloor$.
For $0\le k\le K_n$ and $u=k\tau_n$ we then have, for all sufficiently large
$n$, $u\le1/2$ and $nu^2\le1/4$. Hence
\eqref{eq:lambda-bn-lower-intermediate} and \eqref{eq:bessel-estimates} yield
\begin{align}
  \operatorname{Re}\Lambda_n(b_n^k)
  &\ge\frac12\exp\bigl(n(u-u^2)\bigr)\notag\\
  &\ge \frac12e^{-1/4}n^k.
  \label{eq:lambda-bn-growth-case1}
\end{align}

We finally pass to the inverse. Since
\[
  q_n:=\frac{\beta_n}{\alpha_n}=\frac1{n+2},
  \qquad
  f_n=\alpha_n(1-q_nb_n),
\]
and $q_n\norm{b_n}_{A(G)}\le n/(n+2)<1$, the Neumann series gives
\[
  g_n:=(1-q_nb_n)^{-1}
  =\sum_{k=0}^{\infty}q_n^kb_n^k,
  \qquad
  f_n^{-1}=\alpha_n^{-1}g_n.
\]
For
\[
  S_{n,K}:=\sum_{k=0}^{K}q_n^kb_n^k,
\]
the exact identity
\[
  g_n=S_{n,K}+q_n^{K+1}b_n^{K+1}g_n
\]
gives
\[
  \norm{S_{n,K}}_{A(G)}
  \le
  \left(
    1+(q_n\norm{b_n}_{A(G)})^{K+1}
  \right)\norm{g_n}_{A(G)}
  \le2\norm{g_n}_{A(G)}.
\]
Since $\norm{\Lambda_n}\le1$ by \eqref{eq:lambda-contractive-case1},
\begin{equation}
  \norm{g_n}_{A(G)}
  \ge\frac12\abs{\Lambda_n(S_{n,K})}.
  \label{eq:G-lower-negative}
\end{equation}
Taking $K=K_n$ and using \eqref{eq:lambda-bn-growth-case1}, we obtain for all
sufficiently large $n$
\begin{align*}
  \abs{\Lambda_n(S_{n,K_n})}
  &\ge \frac12e^{-1/4}
  \sum_{k=0}^{K_n}\left(\frac{n}{n+2}\right)^k\\
  &\ge \frac14e^{-1/4}(K_n+1),
\end{align*}
where we used $K_n/n\to0$. Together with \eqref{eq:G-lower-negative} this shows
that
\[
  \norm{f_n^{-1}}_{A(G)}\longrightarrow\infty.
\]

\subsubsection{Infinite two-torsion: an exact dyadic construction}

Assume now that $\widehat G[2]$ is infinite and choose a sequence
$\gamma_1,\gamma_2,\ldots\in\widehat G[2]$ which is linearly independent over
$\mathbb F_2$. Set
\[
  d_n:=\frac1n\sum_{j=1}^n\gamma_j,
  \qquad
  r_n:=\log n,
  \qquad
  \tau_n:=\frac{\log n}{n},
  \qquad
  b_n:=\exp(i r_nd_n).
\]
Since $\gamma_j^2=1$, each $\gamma_j$ takes only the values $\pm1$, and therefore
for every $u\in\R$
\begin{equation}
  e^{iu\gamma_j}
  =\cos u+i\sin u\,\gamma_j.
  \label{eq:order-two-exponential}
\end{equation}
It follows that
\begin{equation}
  b_n
  =\prod_{j=1}^n
  \bigl(\cos\tau_n+i\sin\tau_n\,\gamma_j\bigr).
  \label{eq:bn-dyadic-product}
\end{equation}
As in the preceding case, $d_n$ is real-valued, hence
\[
  \abs{b_n(x)}=1
  \quad(x\in G),
  \qquad
  \norm{b_n}_{A(G)}\le n.
\]
With the same choices
\[
  \beta_n=\frac1{2(n+1)},
  \qquad
  \alpha_n=\frac12+\beta_n,
  \qquad
  f_n=\alpha_n-\beta_nb_n,
\]
we again obtain
\begin{equation}
  \norm{f_n}_{A(G)}\le1,
  \qquad
  \inf_{x\in G}\abs{f_n(x)}\ge\frac12.
  \label{eq:normalization-case2}
\end{equation}

Let
\[
  H_n:=\operatorname{span}_{\mathbb F_2}\{\gamma_1,\ldots,\gamma_n\}
  \subset\widehat G.
\]
Every $\gamma\in H_n$ has a unique representation
$\gamma=\prod_{j\in I}\gamma_j$. Define
\[
  \omega_n(\gamma):=i^{-\abs I}
  \quad(\gamma\in H_n),
  \qquad
  \omega_n(\gamma):=0
  \quad(\gamma\notin H_n),
\]
and
\[
  \Lambda_n(f)
  :=\sum_{\gamma\in\widehat G}
  \omega_n(\gamma)\widehat f(\gamma).
\]
Again $\norm{\Lambda_n}\le1$.

For $k\ge0$ and $u=k\tau_n$, formula \eqref{eq:order-two-exponential} gives the
exact factorization
\[
  b_n^k
  =\prod_{j=1}^n
  \bigl(\cos u+i\sin u\,\gamma_j\bigr).
\]
For $0\le u\le1/2$ the functional $\Lambda_n$ removes all phases, and linear
independence over $\mathbb F_2$ gives the exact identity
\begin{equation}
  \Lambda_n(b_n^k)=(\cos u+\sin u)^n.
  \label{eq:dyadic-exact}
\end{equation}
There is no Bessel tail in this case. The elementary bounds
\[
  \cos u\ge1-\frac{u^2}{2},
  \qquad
  \sin u\ge u-\frac{u^3}{6}
\]
give
\[
  \cos u+\sin u
  \ge
  1+t,
  \qquad
  t:=u-\frac{u^2}{2}-\frac{u^3}{6}.
\]
For $0\le u\le1/2$ we have $0\le t\le1$, and
$\log(1+t)\ge t-t^2/2$. Hence
\[
  \log(\cos u+\sin u)
  \ge
  t-\frac{t^2}{2}.
\]
A direct expansion shows
\[
  t-\frac{t^2}{2}-(u-2u^2)
  =
  \frac{u^2}{72}
  \left(
    72+24u+3u^2-6u^3-u^4
  \right)\ge0
\]
on $[0,1/2]$. Therefore
\begin{equation}
  \log(\cos u+\sin u)\ge u-2u^2
  \qquad(0\le u\le1/2).
  \label{eq:dyadic-log-bound}
\end{equation}
For $0\le k\le K_n$, where $K_n$ is given by \eqref{eq:Kn-negative}, we thus have
\begin{equation}
  \Lambda_n(b_n^k)
  \ge\exp\bigl(n(u-2u^2)\bigr)
  \ge e^{-1/2}n^k.
  \label{eq:dyadic-growth}
\end{equation}

The remainder of the argument is identical to the last part of the preceding case.
Indeed,
\[
  f_n=\alpha_n(1-q_nb_n),
  \qquad
  q_n=\frac1{n+2},
\]
and the Neumann series together with the finite-sum identity used in
\eqref{eq:G-lower-negative}, combined with $K_n\to\infty$, yields
\[
  \norm{f_n^{-1}}_{A(G)}\longrightarrow\infty.
\]

\begin{proof}[Proof of Theorem~\ref{thm:sharp-half-compact}]
Apply Lemma~\ref{lem:structural-dichotomy} to $\Gamma=\widehat G$. If
$\widehat G[2]$ is finite, the Bessel construction above gives a sequence satisfying
\eqref{eq:sharp-half-main}. If $\widehat G[2]$ is infinite, the dyadic construction
gives the same conclusion.
\end{proof}

\begin{corollary}
\label{cor:critical-threshold-compact}
For every infinite compact Abelian group $G$, the critical threshold for
norm-controlled inversion in $A(G)$ is exactly $1/2$: norm control holds for every
$\delta>1/2$ by Theorem~\ref{thm:uniform-fourier-inversion}, while it fails at
$\delta=1/2$ by Theorem~\ref{thm:sharp-half-compact}. Since the admissible class
only enlarges as $\delta$ decreases, failure at $1/2$ also implies failure for every
$0<\delta<1/2$.
\end{corollary}

\subsection{Unitized group algebras on nondiscrete groups}
\label{subsec:negative-unitized-L1}

We now turn to the second sharpness mechanism, for the standard unitization of the
group algebra $L^1(H)$, where $H$ is a nondiscrete locally compact Abelian group. We
identify this unitization with
\[
  L^1(H)^\#=\C\delta_0\oplus_1 L^1(H)\subset M(H),
\]
so that
\[
  \norm{\lambda\delta_0+h}_{L^1(H)^\#}
  =\abs{\lambda}+\norm{h}_1.
\]
Here $\delta_0$ is the unit for convolution. The estimate above the critical threshold
is immediate in this setting. Indeed, if
\[
  \mu=\lambda\delta_0+h\in L^1(H)^\#,
  \qquad
  \norm{\mu}_{L^1(H)^\#}\le1,
\]
and
\[
  \inf_{\chi\in\widehat H}\abs{\widehat\mu(\chi)}\ge\delta>\frac12,
\]
then the Riemann--Lebesgue lemma, in the form
\cite[Theorem~1.2.4(a)]{Rudin1990}, gives $\widehat h\in C_0(\widehat H)$. Since
$H$ is nondiscrete, $\widehat H$ is noncompact by
\cite[Theorem~1.2.5]{Rudin1990}, and therefore
\[
  \abs{\lambda}
  =\lim_{\chi\to\infty}\abs{\lambda+\widehat h(\chi)}
  \ge\delta.
\]
The algebra $L^1(H)^\#$ splits at the unit, so
Lemma~\ref{lem:nikolski-splitting} immediately yields the following estimate.

\begin{proposition}
\label{prop:L1-positive-half}
Let $H$ be a nondiscrete locally compact Abelian group and let
$\mu\in L^1(H)^\#$ satisfy
\[
  \norm{\mu}_{L^1(H)^\#}\le1,
  \qquad
  \inf_{\chi\in\widehat H}\abs{\widehat\mu(\chi)}\ge\delta>\frac12.
\]
Then $\mu$ is invertible in $L^1(H)^\#$ and
\begin{equation}
  \norm{\mu^{-1}}_{L^1(H)^\#}
  \le\frac{1}{2\delta-1}.
  \label{eq:L1-positive-bound}
\end{equation}
\end{proposition}

Thus only sharpness at the endpoint $\delta=1/2$ remains to be proved. The
construction below is the locally compact counterpart of Subsection~\ref{subsec:negative-compact-fourier}. The
main new ingredient is a \emph{smearing procedure}: a finitely supported element
of the discrete group algebra of the underlying abstract group is replaced by an
$L^1$-function supported on small disjoint translates of a neighborhood of the
identity. This preserves exactly the values of the auxiliary functional on all
convolution powers that are used to force the inverse norm to grow.

\subsubsection{A finitely supported discrete block}

We first extract, by passing to Fourier coefficients, the finite block encoded in the
$A(G)$-constructions of Subsection~\ref{subsec:negative-compact-fourier}. If $\Gamma$ is an Abelian group,
$\Gamma_d$ denotes the same group with the discrete topology.

\begin{lemma}[Finitely supported discrete block]
\label{lem:finite-discrete-block-L1}
There is an absolute constant
\[
  c_*:=\frac{e^{-1/2}}8>0
\]
with the following property. Let $\Gamma$ be an infinite Abelian group, and let
$(K_n)$ be any sequence of positive integers satisfying \eqref{eq:Kn-negative}.
For every sufficiently large $n$ there exist a finitely supported element
$\nu_n\in\ell^1(\Gamma)$ and a function $\omega_n:\Gamma\to\C$ such that
$\abs{\omega_n}\le1$ and, with
\[
  \Lambda_n(a):=\sum_{\gamma\in\Gamma}\omega_n(\gamma)a(\gamma),
\]
one has
\begin{align}
  \norm{\nu_n}_1&\le n,
  \label{eq:finite-block-norm}\\
  \sup_{\chi\in\widehat{\Gamma_d}}
  \abs{\widehat{\nu_n}(\chi)}&\le1,
  \label{eq:finite-block-transform}\\
  \operatorname{Re}\Lambda_n(\nu_n^{*k})&\ge c_*n^k,
  \qquad 0\le k\le K_n.
  \label{eq:finite-block-powers}
\end{align}
Moreover, $\norm{\Lambda_n}\le1$ and $\Lambda_n(\delta_0)=1$.
\end{lemma}

\begin{proof}
The two constructions are those of Subsection~\ref{subsec:negative-compact-fourier}, with one additional
truncation in the Bessel case in order to make the support finite.

If $\Gamma[2]$ is infinite, choose $x_1,\ldots,x_n\in\Gamma[2]$ linearly
independent over $\mathbb F_2$, put
\[
  \tau_n:=\frac{\log n}{n},
\]
and take
\begin{equation}
  \nu_n
  :=\mathop{\ast}_{j=1}^n
  \bigl(\cos\tau_n\,\delta_0+i\sin\tau_n\,\delta_{x_j}\bigr).
  \label{eq:finite-block-dyadic}
\end{equation}
This element is finitely supported. Since the subset sums of
$x_1,\ldots,x_n$ are distinct, the product formula gives
\[
  \norm{\nu_n}_1
  =(\cos\tau_n+\sin\tau_n)^n
  \le e^{n\tau_n}=n,
\]
which proves \eqref{eq:finite-block-norm} in this case. The phase-correcting
functional constructed in Subsection~\ref{subsec:negative-compact-fourier} gives, for $u=k\tau_n$,
\[
  \Lambda_n(\nu_n^{*k})=(\cos u+\sin u)^n,
\]
and the estimates leading to \eqref{eq:dyadic-growth} imply
\eqref{eq:finite-block-powers}. Finally, \eqref{eq:finite-block-transform}
follows from the same product formula.

Suppose now that $\Gamma[2]$ is finite. Choose a strongly dissociated family
$x_1,\ldots,x_n$ as in Lemma~\ref{lem:structural-dichotomy} and define
\[
  d_n:=\frac1{2n}\sum_{j=1}^n
  (\delta_{x_j}+\delta_{-x_j}),
  \qquad
  b_n:=\exp(i\log n\,d_n).
\]
The construction in Subsection~\ref{subsec:negative-compact-fourier} gives
\[
  \abs{\widehat b_n(\chi)}=1
  \quad(\chi\in\widehat{\Gamma_d}),
  \qquad
  \norm{b_n}_1\le n,
\]
and the corresponding phase-correcting functional satisfies a lower bound of
order $n^k$ for $0\le k\le K_n$. To obtain finite support, we truncate the exponential series and set
\[
  v_n:=\sum_{m=0}^{n}
  \frac{(i\log n)^m}{m!}\,d_n^{*m}
\]
and
\begin{equation}
  \eta_n:=\sum_{m=n+1}^{\infty}\frac{(\log n)^m}{m!}
  =n-\sum_{m=0}^{n}\frac{(\log n)^m}{m!}.
  \label{eq:eta-L1}
\end{equation}
Then
\[
  \norm{v_n-b_n}_1\le\eta_n,
  \qquad
  \eta_n
  \le2\left(\frac{e\log n}{n+1}\right)^{n+1}
\]
for all sufficiently large $n$. In particular, $\eta_n$ decays faster than any
negative power of $n$. Define the finitely supported element
\begin{equation}
  \nu_n:=\frac{v_n}{1+\eta_n}.
  \label{eq:finite-block-bessel}
\end{equation}
Since $\abs{\widehat b_n}=1$ and
$\norm{v_n-b_n}_1\le\eta_n$, we obtain
\eqref{eq:finite-block-transform}; moreover,
\[
  \norm{\nu_n}_1
  \le\frac{n+\eta_n}{1+\eta_n}\le n.
\]
For $1\le k\le K_n$, telescoping powers gives
\[
  \norm{v_n^{*k}-b_n^{*k}}_1
  \le k\eta_n(n+\eta_n)^{k-1}.
\]
After division by $n^k$,
\begin{equation}
  \frac{\norm{v_n^{*k}-b_n^{*k}}_1}{n^k}
  \le
  \frac{k\eta_n}{n}
  \left(1+\frac{\eta_n}{n}\right)^{k-1}
  \le
  \frac{K_n\eta_n}{n}
  \left(1+\frac{\eta_n}{n}\right)^{K_n}
  =o(1),
  \label{eq:finite-block-uniform-truncation}
\end{equation}
uniformly for $1\le k\le K_n$. Also
\[
  1-(1+\eta_n)^{-k}
  \le k\eta_n
  \le K_n\eta_n=o(1)
\]
uniformly in the same range. Using \eqref{eq:lambda-bn-growth-case1},
\begin{align*}
  \frac{\operatorname{Re}\Lambda_n(\nu_n^{*k})}{n^k}
  &=
  (1+\eta_n)^{-k}
  \frac{\operatorname{Re}\Lambda_n(v_n^{*k})}{n^k}\\
  &\ge
  (1+\eta_n)^{-k}
  \left(
    \frac12e^{-1/4}
    -
    \frac{\norm{v_n^{*k}-b_n^{*k}}_1}{n^k}
  \right).
\end{align*}
Since
\[
  \frac12e^{-1/4}>\frac{e^{-1/2}}8=c_*,
\]
the right-hand side is at least $c_*$ for every $1\le k\le K_n$ once $n$ is
sufficiently large. The case $k=0$ is immediate from
$\Lambda_n(\delta_0)=1$. This proves \eqref{eq:finite-block-powers} with the
stated constant.
\end{proof}

\subsubsection{Smearing the discrete block}

We now use the topology of $H$. The following lemma is the step that converts the
finitely supported discrete construction into an element of $L^1(H)$ without
losing the information carried by the first $K$ convolution powers.

\begin{lemma}[Exact smearing]
\label{lem:exact-smearing-L1}
Let $H$ be a nondiscrete locally compact Abelian group. Let $\nu$ be finitely
supported, set $E:=\supp\nu$, and write
\[
  \nu=\sum_{\gamma\in E}a_\gamma\delta_\gamma.
\]
Let $K\ge1$ be an integer and let $\omega:H\to\C$ satisfy
$\abs{\omega}\le1$. Then one can choose a compact symmetric neighborhood
$C$ of $0$ with positive Haar measure and put
\[
  u:=\frac{\mathbf1_C}{m_H(C)},
  \qquad
  h:=\nu*u
  =\frac1{m_H(C)}\sum_{\gamma\in E}
  a_\gamma\mathbf1_{\gamma+C}\in L^1(H),
\]
so that there is a functional
\[
  \widetilde\Lambda\in(L^1(H)^\#)^*,
  \qquad
  \norm{\widetilde\Lambda}\le1,
\]
for which
\begin{align}
  \norm h_1&=\norm\nu_1,
  \label{eq:smearing-norm-L1}\\
  \abs{\widehat h(\chi)}&\le\abs{\widehat\nu(\chi)}
  \qquad(\chi\in\widehat H),
  \label{eq:smearing-transform-L1}\\
  \widetilde\Lambda(h^{*k})
  &=\sum_{\gamma\in H}\omega(\gamma)\nu^{*k}(\{\gamma\}),
  \qquad 1\le k\le K,
  \label{eq:smearing-powers-L1}\\
  \widetilde\Lambda(\delta_0)&=1.
  \label{eq:smearing-unit-L1}
\end{align}
\end{lemma}

\begin{proof}
Set
\[
  \mathcal F:=\bigcup_{k=1}^{K}\supp(\nu^{*k}),
  \qquad
  D:=(\mathcal F-\mathcal F)\setminus\{0\}.
\]
Both sets are finite, and $0\notin D$. For a positive integer $k$ and a subset
$C\subset H$, we use the standard additive notation
\[
  kC:=\underbrace{C+\cdots+C}_{k\text{ times}}.
\]
Choose a neighborhood $V$ of $0$ such that $V\cap D=\varnothing$. By continuity of
addition there exists a symmetric neighborhood $U$ of $0$ such that $(2K)U\subset V$.
By local compactness, $U$ contains a compact symmetric neighborhood $C$ of $0$. Then
$m_H(C)>0$ and
\begin{equation}
  (2K)C\cap D=\varnothing.
  \label{eq:smearing-separation-L1}
\end{equation}
Then the sets $y+KC$, $y\in\mathcal F$, are pairwise disjoint. In particular,
the translates $\gamma+C$, $\gamma\in E$, are disjoint, which gives
\eqref{eq:smearing-norm-L1}. Also
\[
  \widehat h=\widehat\nu\,\widehat u,
  \qquad
  \abs{\widehat u}\le\norm u_1=1,
\]
so \eqref{eq:smearing-transform-L1} follows.

For $1\le k\le K$ we have
\[
  h^{*k}
  =\sum_{y\in\supp(\nu^{*k})}
  \nu^{*k}(\{y\})u^{*k}(\,\cdot-y\,),
\]
and the support of $u^{*k}(\,\cdot-y\,)$ is contained in $y+kC\subset y+KC$.
Define
\[
  \Phi(\xi):=
  \begin{cases}
    \omega(y),&\xi\in y+KC\text{ for some }y\in\mathcal F,\\
    0,&\text{otherwise}.
  \end{cases}
\]
The sets $y+KC$, $y\in\mathcal F$, form a finite pairwise disjoint family of compact
sets. Hence $\Phi$ is Borel measurable, belongs to $L^\infty(H)$, is well-defined by
\eqref{eq:smearing-separation-L1}, and satisfies $\norm{\Phi}_\infty\le1$. Since $u\ge0$ and $\int_Hu\,dm_H=1$, one has
$\int_Hu^{*k}\,dm_H=1$. Therefore
\[
  \int_H\Phi(\xi)h^{*k}(\xi)\,dm_H(\xi)
  =\sum_y\omega(y)\nu^{*k}(\{y\}).
\]
Finally set
\[
  \widetilde\Lambda(\lambda\delta_0+g)
  :=\lambda+\int_H\Phi g\,dm_H.
\]
Because the unitization is the $\ell^1$-direct sum
$\C\delta_0\oplus_1L^1(H)$, this functional has norm one and satisfies
\eqref{eq:smearing-powers-L1}--\eqref{eq:smearing-unit-L1}.
\end{proof}

\subsubsection{The endpoint construction}

We can now prove sharpness. The construction is explicit up to the choice of the
small neighborhood $C_n$ in Lemma~\ref{lem:exact-smearing-L1}.

\begin{theorem}[Sharpness at $1/2$ for the unitized group algebra]
\label{thm:sharp-half-unitalized-L1}
Let $H$ be a nondiscrete locally compact Abelian group. There exists a sequence of
invertible elements $(\mu_n)\subset L^1(H)^\#$ such that
\begin{equation}
  \norm{\mu_n}_{L^1(H)^\#}\le1,
  \qquad
  \inf_{\chi\in\widehat H}\abs{\widehat\mu_n(\chi)}\ge\frac12,
  \qquad
  \norm{\mu_n^{-1}}_{L^1(H)^\#}\longrightarrow\infty.
  \label{eq:L1-sharp-main}
\end{equation}
\end{theorem}

\begin{proof}
The underlying abstract group of $H$ is infinite. Apply
Lemma~\ref{lem:finite-discrete-block-L1} to obtain $\nu_n$, $\omega_n$ and
$K_n$, and then apply Lemma~\ref{lem:exact-smearing-L1} with $K=K_n$. We obtain
$h_n\in L^1(H)$ and a contractive functional $\widetilde\Lambda_n$ satisfying
\begin{equation}
  \norm{h_n}_1\le n,
  \qquad
  \sup_{\chi\in\widehat H}\abs{\widehat h_n(\chi)}\le1,
  \qquad
  \operatorname{Re}\widetilde\Lambda_n(h_n^{*k})\ge c_*n^k
  \quad(0\le k\le K_n).
  \label{eq:hn-data-L1}
\end{equation}
The second estimate uses the fact that the continuous characters of $H$ are, in
particular, abstract characters of $H_d$.

Put
\[
  \beta_n:=\frac1{2(n+1)},
  \qquad
  \alpha_n:=\frac12+\beta_n=\frac{n+2}{2(n+1)},
\]
and define
\begin{equation}
  \mu_n:=\alpha_n\delta_0-\beta_n h_n.
  \label{eq:mu-L1-construction}
\end{equation}
Then, by \eqref{eq:hn-data-L1},
\[
  \norm{\mu_n}_{L^1(H)^\#}
  =\alpha_n+\beta_n\norm{h_n}_1
  \le\alpha_n+n\beta_n=1,
\]
while for every $\chi\in\widehat H$,
\[
  \abs{\widehat\mu_n(\chi)}
  \ge\alpha_n-\beta_n\abs{\widehat h_n(\chi)}
  \ge\alpha_n-\beta_n
  =\frac12.
\]
Thus $\mu_n$ is admissible at the endpoint $\delta=1/2$.

It remains to estimate the inverse from below. Write
\[
  \mu_n=\alpha_n(\delta_0-q_nh_n),
  \qquad
  q_n:=\frac{\beta_n}{\alpha_n}=\frac1{n+2}.
\]
Since
\[
  q_n\norm{h_n}_1\le\frac{n}{n+2}<1,
\]
$\mu_n$ is invertible and
\[
  G_n:=(\delta_0-q_nh_n)^{-1}
  =\sum_{k=0}^{\infty}q_n^kh_n^{*k}.
\]
For a finite $K$ set
\[
  S_{n,K}:=\sum_{k=0}^{K}q_n^kh_n^{*k}.
\]
Rather than applying $\widetilde\Lambda_n$ directly to the infinite Neumann
series, we use the exact resolvent identity
\[
  G_n
  =S_{n,K}+q_n^{K+1}h_n^{*(K+1)}*G_n.
\]
Consequently,
\[
  \norm{S_{n,K}}
  \le\left(1+(q_n\norm{h_n}_1)^{K+1}\right)\norm{G_n}
  \le2\norm{G_n},
\]
and hence
\begin{equation}
  \norm{G_n}
  \ge\frac12\abs{\widetilde\Lambda_n(S_{n,K})}.
  \label{eq:G-lower-L1}
\end{equation}
Take $K=K_n$. From \eqref{eq:hn-data-L1},
\begin{align*}
  \abs{\widetilde\Lambda_n(S_{n,K_n})}
  &\ge\operatorname{Re}\widetilde\Lambda_n(S_{n,K_n})\\
  &\ge c_*\sum_{k=0}^{K_n}q_n^kn^k\\
  &=c_*\sum_{k=0}^{K_n}
  \left(\frac{n}{n+2}\right)^k.
\end{align*}
Since $K_n/n\to0$, for all sufficiently large $n$ every term in the last sum
is at least $1/2$. Therefore
\[
  \abs{\widetilde\Lambda_n(S_{n,K_n})}
  \ge\frac{c_*}{2}(K_n+1).
\]
Combining this with \eqref{eq:G-lower-L1},
$\mu_n^{-1}=\alpha_n^{-1}G_n$, and $K_n\to\infty$, we obtain
\[
  \norm{\mu_n^{-1}}_{L^1(H)^\#}\longrightarrow\infty.
\]
\end{proof}

\begin{corollary}
\label{cor:critical-threshold-unitalized-L1}
For every nondiscrete locally compact Abelian group $H$, the critical threshold for
norm-controlled inversion in $L^1(H)^\#$, measured on the dual group $\widehat H$,
is exactly $1/2$. More precisely, Proposition~\ref{prop:L1-positive-half} gives
norm control for every $\delta>1/2$, whereas
Theorem~\ref{thm:sharp-half-unitalized-L1} shows that no uniform control is
possible at $\delta=1/2$.
\end{corollary}

\begin{remark}
For a discrete group $H$, the unit is already contained in $L^1(H)=\ell^1(H)$.
Since $\widehat H$ is compact by \cite[Theorem~1.2.5]{Rudin1990}, the Fourier
transform identifies $L^1(H)$ isometrically with $A(\widehat H)$; thus this case is
exactly the Fourier-algebra problem treated in Subsection~\ref{subsec:negative-compact-fourier}. Consequently, in the
infinite cases, Corollaries~\ref{cor:critical-threshold-compact} and
\ref{cor:critical-threshold-unitalized-L1} together cover both sides of the
compact/discrete duality. The threshold and the sharp upper majorant
$(2\delta-1)^{-1}$ for the nondiscrete unitized group algebra already appear in
\cite{Nikolski1999}; the point of the construction above is that the failure at the
endpoint is exhibited by an explicit sequence obtained from finite discrete blocks
and exact smearing.
\end{remark}

\section{Further remarks and open problems}
\label{sec:further-questions}

\begin{enumerate}

\item \textbf{Effective bounds in the full range.}
The compactness arguments in the proofs of
Theorems~\ref{thm:uniform-fourier-inversion} and
\ref{thm:uniform-measure-inversion} show that group-independent bounds exist for
every $\delta>1/2$, but they do not make those bounds effective throughout the
full range. Can one give an explicit upper bound, depending only on $\delta$, for
$\norm{f^{-1}}_{A(K)}$ whenever $K$ is compact Abelian,
$\norm f_{A(K)}\le1$, and $\inf_K\abs f\ge\delta$, and similarly an explicit
upper bound for $\norm{\mu^{-1}}_{M(G)}$ uniformly over all locally compact Abelian
groups $G$ under the assumptions of Theorem~B? This should be distinguished from
the explicit estimates available on narrower ranges: Nikolski's bound
$(2\delta^2-1)^{-1}$ applies universally for $\delta>1/\sqrt2$, and in the connected
Fourier-algebra case Theorem~\ref{thm:inverse-bound-connected} gives an explicit
bound for $\delta>2-\sqrt2$. Obtaining effective estimates in the newly covered
part of the range $\delta>1/2$ is a natural quantitative problem.

\item \textbf{The critical level for a large Fourier coefficient on the circle.}
Subsection~\ref{subsec:largest-coefficient} shows that sufficiently strong pointwise
separation from zero forces a large Fourier coefficient, while
Example~\ref{ex:no-large-fourier-coefficient} shows that this mechanism cannot by
itself reach the inversion threshold $1/2$. We formulate the corresponding extremal
question here specifically for the classical Wiener algebra on the circle. Let
$\delta_{\mathrm{fc}}(\T)$ be the infimum of those $1/2<\delta\le1$ for which every
$f\in A(\T)$ satisfying
\[
  \norm{f}_{A(\T)}\le1,
  \qquad
  \inf_{z\in\T}\abs{f(z)}\ge\delta
\]
necessarily has a Fourier coefficient of modulus strictly greater than $1/2$, that is,
$\norm{\fhat}_{\infty}>1/2$. Theorem~\ref{thm:improved-largest-coefficient} gives the
upper bound $\delta_{\mathrm{fc}}(\T)\le 2-\sqrt2$.

The lower bound can be improved by optimizing the parameter in the family from
Example~\ref{ex:no-large-fourier-coefficient}. For $0<a<1$ and $z\in\T$, put
\[
  B_a(z):=\frac{z-a}{1-az}
  =-a+(1-a^2)\sum_{n=1}^{\infty}a^{n-1}z^n,
  \qquad
  f_a(z):=\frac{B_a(z)}{1+2a}.
\]
Since $\abs{B_a}=1$ on $\T$ and
$\norm{B_a}_{A(\T)}=a+(1-a^2)/(1-a)=1+2a$, we have
\[
  \norm{f_a}_{A(\T)}=1,
  \qquad
  \inf_{z\in\T}\abs{f_a(z)}=\frac{1}{1+2a},
  \qquad
  \norm{\widehat{f_a}}_\infty
  =\frac{\max\{a,1-a^2\}}{1+2a}.
\]
The inequality $a/(1+2a)<1/2$ is automatic, whereas
\[
  \frac{1-a^2}{1+2a}\le\frac12
  \quad\Longleftrightarrow\quad
  2a^2+2a-1\ge0
  \quad\Longleftrightarrow\quad
  a\ge a_*:=\frac{\sqrt3-1}{2}.
\]
As $a\mapsto(1+2a)^{-1}$ is decreasing, the largest separation from zero in this
normalized one-factor family under the constraint
$\norm{\widehat{f_a}}_\infty\le1/2$ is attained at $a=a_*$ and equals
\[
  \inf_{\T}\abs{f_{a_*}}=\frac1{\sqrt3},
  \qquad
  \norm{\widehat{f_{a_*}}}_\infty=\frac12.
\]
For $a>a_*$ the largest coefficient is strictly smaller than $1/2$, and the
corresponding lower bounds approach $1/\sqrt3$ as $a\downarrow a_*$. Consequently,
\[
  \frac1{\sqrt3}\le \delta_{\mathrm{fc}}(\T)\le 2-\sqrt2.
\]
Thus, already for $A(\T)$, a direct dominant-coefficient argument applied to the
original function cannot work at or below $1/\sqrt3$, even though norm-controlled
inversion holds throughout the range $\delta>1/2$. The estimates above leave a gap between $1/\sqrt3$ and $2-\sqrt2$ for
$\delta_{\mathrm{fc}}(\T)$. Determining its exact value is a natural next question.
We emphasize that this formulation is deliberately restricted to the circle. For
disconnected compact Abelian groups the situation may be entirely different, and the
corresponding large-coefficient threshold should be regarded as a separate problem.

\item \textbf{Fourier and Fourier--Stieltjes algebras of non-Abelian groups.}
For a general, not necessarily Abelian, locally compact group $G$, it is natural to ask
for analogues of the results of this paper for the Fourier algebra $A(G)$ and the
Fourier--Stieltjes algebra $B(G)$. These algebras are still commutative as function
algebras under pointwise multiplication; the non-Abelian feature lies in the
underlying group and in the representation-theoretic structure of the algebras.
Some first results in this direction were obtained in
\cite[Section~3]{OhryskoWasilewski2020}: for $B(G)$, qualitative inversion was proved
above the level $1/2$, and norm-controlled inversion was obtained for
$\delta>1/\sqrt2$. It is natural to ask whether the threshold $1/2$ is also critical
for norm-controlled inversion in $B(G)$ and to understand the corresponding problem
for Fourier algebras of compact non-Abelian groups.

\item \textbf{Norm-controlled solutions of B\'ezout equations.}
Following \cite[Definition~0.2.3 and Subsection~0.2.4]{Nikolski1999}, consider
$(\mathcal A,X)=(A(K),K)$ for a compact Abelian group $K$, or
$(\mathcal A,X)=(M(G),\widehat G)$ for a locally compact Abelian group $G$.
For $n\ge2$ and $\mathbf f=(f_1,\ldots,f_n)\in\mathcal A^n$, use the normalization
\[
  \norm{\mathbf f}_{\mathcal A^n}
  :=\left(\sum_{j=1}^n\norm{f_j}_{\mathcal A}^2\right)^{1/2}\le1,
  \qquad
  \inf_{x\in X}\left(\sum_{j=1}^n\abs{f_j(x)}^2\right)^{1/2}\ge\delta.
\]
Here, in the measure-algebra case, $f_j(x)$ means the Fourier--Stieltjes transform
evaluated at $x\in\widehat G$. The problem is to solve
\[
  \sum_{j=1}^n f_jg_j=e,
  \qquad \mathbf g=(g_1,\ldots,g_n)\in\mathcal A^n,
\]
with a bound on $\norm{\mathbf g}_{\mathcal A^n}$ depending only on $\delta$ and
$n$. The unit $e$ is $1$ in $A(K)$ and $\delta_0$ in $M(G)$, where the product
is convolution.
For $0<\delta\le1$, let $c_n(\delta;\mathcal A,X)$ be the supremum, over all
such data $\mathbf f$, of
\[
  \inf\left\{\norm{\mathbf g}_{\mathcal A^n}:
       \mathbf g\in\mathcal A^n,\ \sum_{j=1}^n f_jg_j=e\right\},
\]
where the infimum is $+\infty$ if no solution exists, and set
\[
  \delta_n(\mathcal A,X)
  :=\inf\{0<\delta\le1:c_n(\delta;\mathcal A,X)<\infty\}.
\]
Nikolski's estimate \cite[Theorem~2.1.1]{Nikolski1999} gives, for both classes,
\[
  \sup_{n\ge1}c_n(\delta;\mathcal A,X)
  \le\frac{1}{2\delta^2-1},
  \qquad \delta>\frac1{\sqrt2},
\]
uniformly in the group. Padding scalar data with zeros shows that
$c_n\ge c_1$; hence, for infinite $K$ or $G$, respectively, the endpoint
obstruction from Section~\ref{sec:sharpness} also applies to these equations
for every $n$.

Does the conclusion of Theorems~A and~B extend to every fixed $n\ge2$ with the
same threshold $1/2$, with bounds independent of the underlying group? In
particular, is $\delta_n(A(K),K)=1/2$ for every infinite compact Abelian $K$, and
is $\delta_n(M(G),\widehat G)=1/2$ for every infinite locally compact Abelian $G$?
More strongly, for each $\delta>1/2$, can the solution bounds be chosen
independently of both the group and $n$? In Nikolski's terminology this asks for
complete $\delta$-visibility with a group-independent majorant. The usual
reduction to $\sum_j f_jf_j^*$, using complex conjugation for $A(K)$ and the
measure involution for $M(G)$, replaces $\delta$ by $\delta^2$
\cite[Lemma~1.4.2]{Nikolski1999}; thus this reduction does not settle the
range $1/2<\delta\le1/\sqrt2$.

\end{enumerate}

\section*{Use of artificial intelligence}
This work was supported by OpenAI through the ChatGPT for Academic Researchers program.

AI-assisted tools were used during the preparation of this manuscript for mathematical
verification and editorial support. The author has independently verified the content
and takes full responsibility for its accuracy and conclusions.


\begin{thebibliography}{99}

\bibitem{Arens1963}
R.~Arens,
\emph{The group of invertible elements of a commutative Banach algebra},
Studia Math., Special Series \textbf{1} (1963), 21--23.

\bibitem{Beurling1939}
A.~Beurling,
\emph{Sur les int\'egrales de Fourier absolument convergentes et leur application
\`a une transformation fonctionnelle},
in \emph{Neuvi\`eme Congr\`es des Math\'ematiciens Scandinaves
(Helsingfors, 1938)}, Mercators Tryckeri, Helsingfors, 1939, 345--366.

\bibitem{Cohen1961}
P.~J.~Cohen,
\emph{A note on constructive methods in Banach algebras},
Proc. Amer. Math. Soc. \textbf{12} (1961), no.~1, 159--163.

\bibitem{Gelfand1941}
I.~Gelfand,
\emph{Normierte Ringe},
Rec. Math. [Mat. Sbornik] N.S. \textbf{9(51)} (1941), no.~1, 3--24.

\bibitem{GlicksbergWik1972}
I.~Glicksberg and I.~Wik,
\emph{The range of Fourier--Stieltjes transforms of parts of measures},
in D.~Gulick and R.~L.~Lipsman (eds.), \emph{Conference on Harmonic Analysis
(College Park, Maryland, 1971)}, Lecture Notes in Mathematics, vol.~266,
Springer, Berlin--Heidelberg, 1972, 73--77.

\bibitem{Gorin1970}
E.~A.~Gorin,
\emph{A function-algebra variant of a theorem of Bohr--van Kampen},
Math. USSR-Sb. \textbf{11} (1970), no.~2, 233--243.

\bibitem{GrochenigKlotz2013}
K.~Gr\"ochenig and A.~Klotz,
\emph{Norm-controlled inversion in smooth Banach algebras, I},
J. Lond. Math. Soc. (2) \textbf{88} (2013), no.~1, 49--64.

\bibitem{Graham1974}
C.~C.~Graham,
\emph{A Riesz product proof of the Wiener--Pitt theorem},
Proc. Amer. Math. Soc. \textbf{44} (1974), no.~2, 312--314.

\bibitem{Kaniuth2009}
E.~Kaniuth,
\emph{A Course in Commutative Banach Algebras},
Graduate Texts in Mathematics, vol.~246, Springer, New York, 2009.

\bibitem{Newman1975}
D.~J.~Newman,
\emph{A simple proof of Wiener's $1/f$ theorem},
Proc. Amer. Math. Soc. \textbf{48} (1975), no.~1, 264--265.

\bibitem{Nikolskaia2001}
L.~N.~Nikolskaia,
\emph{Estimates of inverses and finite section inverses of Toeplitz operators on the Wiener algebra},
Acta Sci. Math. (Szeged) \textbf{67} (2001), no.~1--2, 395--409.

\bibitem{Nikolski1999}
N.~Nikolski,
\emph{In search of the invisible spectrum},
Ann. Inst. Fourier (Grenoble) \textbf{49} (1999), no.~6, 1925--1998.

\bibitem{Ohrysko2026}
P.~Ohrysko,
\emph{Inversion problem in algebras of integrable functions with summable Fourier transforms},
Math. Nachr. (2026), e70240.

\bibitem{OhryskoSandersWojciechowski2026}
P.~Ohrysko, T.~Sanders, and M.~Wojciechowski,
\emph{Wiener--Pitt sets for compact Abelian groups},
J. Fourier Anal. Appl. \textbf{32} (2026), no.~2, Paper No.~30.

\bibitem{OhryskoWasilewski2020}
P.~Ohrysko and M.~Wasilewski,
\emph{Inversion problem in measure and Fourier--Stieltjes algebras},
J. Funct. Anal. \textbf{278} (2020), no.~5, 108399.

\bibitem{NIST2010}
F.~W.~J.~Olver, D.~W.~Lozier, R.~F.~Boisvert, and C.~W.~Clark (eds.),
\emph{NIST Handbook of Mathematical Functions},
Cambridge University Press, New York, 2010.

\bibitem{Rudin1990}
W.~Rudin,
\emph{Fourier Analysis on Groups},
Wiley Classics Library, John Wiley \& Sons, Inc., New York, 1990,
reprint of the 1962 original.

\bibitem{SameiShepelska2019}
E.~Samei and V.~Shepelska,
\emph{Norm-controlled inversion in weighted convolution algebras},
J. Fourier Anal. Appl. \textbf{25} (2019), no.~6, 3018--3044.

\bibitem{Shapiro1979}
H.~S.~Shapiro,
\emph{A counterexample in harmonic analysis},
Banach Center Publ. \textbf{4} (1979), no.~1, 233--236.

\bibitem{Shreider1950}
Yu.~A.~\v{S}re\u{\i}der,
\emph{The structure of maximal ideals in rings of measures with convolution},
Mat. Sb. (N.S.) \textbf{27(69)} (1950), no.~2, 297--318;
English transl., Amer. Math. Soc. Translation (1953), no.~81, 28~pp.

\bibitem{Stafney1967}
J.~D.~Stafney,
\emph{An unbounded inverse property in the algebra of absolutely convergent Fourier series},
Proc. Amer. Math. Soc. \textbf{18} (1967), no.~3, 497--498.

\bibitem{Wiener1932}
N.~Wiener,
\emph{Tauberian theorems},
Ann. of Math. (2) \textbf{33} (1932), no.~1, 1--100.

\bibitem{WienerPitt1938}
N.~Wiener and H.~R.~Pitt,
\emph{On absolutely convergent Fourier--Stieltjes transforms},
Duke Math. J. \textbf{4} (1938), no.~2, 420--436.

\bibitem{Williamson1959}
J.~H.~Williamson,
\emph{A theorem on algebras of measures on topological groups},
Proc. Edinburgh Math. Soc. (2) \textbf{11} (1959), no.~4, 195--206.

\bibitem{Zelazko1973}
W.~\.{Z}elazko,
\emph{Banach Algebras},
translated from the Polish by M.~E.~Kuczma, Elsevier Publishing Co.,
Amsterdam--London--New York, and PWN--Polish Scientific Publishers,
Warsaw, 1973.

\end{thebibliography}
\end{document}